\documentclass[11pt, a4paper]{amsart}

\usepackage{amsmath}
\usepackage{amstext}
\usepackage{amssymb}
\usepackage{amsthm}
\usepackage{mathabx}
\usepackage{enumerate}
\usepackage{enumitem}
\usepackage{mathrsfs}
\usepackage[all]{xy}
\usepackage{datetime}

\usepackage{stackengine}
\stackMath
\newcommand\bfdelta{%
 \def\useanchorwidth{T}%
    \stackunder[1pt]{\delta}{\scriptscriptstyle\sim}%
}

\usepackage[svgnames]{xcolor}
\usepackage{tikz}

\usepackage[textsize=footnotesize,color=yellow, bordercolor=white]{todonotes}

\pgfdeclarelayer{edgelayer}
\pgfdeclarelayer{nodelayer}
\pgfsetlayers{edgelayer,nodelayer,main}

\tikzstyle{none}=[inner sep=0pt]

\usepackage[pdfauthor={Karen Bakke Haga, David Schrittesser, Asger Toernquist},
            pdftitle={Maximal almost disjoint families, Determinacy, and Forcing},
            pdfproducer={Latex with hyperref}]{hyperref}

\makeatletter
\def\imod#1{\allowbreak\mkern10mu({\operator@font mod}\,\,#1)}
\makeatother

\newtheorem{theorem}{Theorem}[section]

\newtheorem{mprop}[theorem]{Main Proposition}
\newtheorem{lemma}[theorem]{Lemma}
\newtheorem{blemma}[theorem]{The Branch Lemma}
\newtheorem{claim}[theorem]{Claim}
\newtheorem{corollary}[theorem]{Corollary}

\theoremstyle{definition}
\newtheorem{definition}[theorem]{Definition}

\newtheorem{question}[theorem]{Question}
\theoremstyle{remark}
\newtheorem{remark}[theorem]{Remark}
\newtheorem{assumption}[theorem]{Assumption}

\newtheorem{fact}[theorem]{Fact}
\newtheorem{facts}[theorem]{Facts}
\theoremstyle{remark}

\newtheorem{notation}[theorem]{Notation}

\numberwithin{equation}{section}

    \DeclareMathOperator{\lh}{lh}

    \DeclareMathOperator{\dom}{dom}

    \DeclareMathOperator{\I}{I}

    \DeclareMathOperator{\fin}{Fin}

\providecommand{\conc}{ \mathbin{{}^\frown}}

    \newcommand{\ON}{\mathbb{ON}}

    \newcommand{\forces}{\Vdash}

    \newcommand{\ZF}{\mathrm{ZF}}
        \newcommand{\ZFC}{\mathrm{ZFC}}

\def\R{{\mathbb R}}

\def\P{{\mathbb P}}

\def\eL{{\mathbf{L}}}
\def\AD{{\ensuremath{\mathbf{AD}}}}
\def\PD{{\ensuremath{\mathbf{PD}}}}
\newcommand{\initseg}{\sqsubseteq}

\newcommand{\setdef}{\;\vert\;}
\providecommand{\res}{\mathbin{\upharpoonright}}

\DeclareMathOperator{\powerset}{\mathcal{P}}

\newcommand{\notforces}{\nVdash}
\def\M{{\mathbf M}}
\def\A{{\mathcal A}}
\def\I{{\mathcal I}}

\title[MAD families, PD, and forcing]{Maximal almost disjoint families, Determinacy, and Forcing}

\author{Karen Bakke Haga, David Schrittesser, Asger T\"ornquist}

\address{Department of Mathematical Sciences, University of Copenhagen, Universitetsparken 5, 2100 Copenhagen, Denmark}

\address{Kurt G\"odel Research Center, Universit\"at Wien, W\"ahringerstra\ss e 26, 1090 Vienna, Austria}

\address{Department of Mathematical Sciences, University of Copenhagen, Universitetsparken 5, 2100 Copenhagen, Denmark}

\email{karenbhaga@math.ku.dk}

\email{david@logic.univie.ac.at}

\email{asgert@math.ku.dk}


\date{\today}


\begin{document}

\begin{abstract}
We study the notion of $\mathcal J$-MAD families where $\mathcal J$ is a Borel ideal on $\omega$. We show that if $\mathcal J$ is an arbitrary $F_\sigma$ ideal, or is any finite or countably iterated Fubini product of $F_\sigma$ ideals, then there are no analytic infinite $\mathcal J$-MAD families, and assuming Projective Determinacy there are no infinite projective $\mathcal J$-MAD families; and under the full Axiom of Determinacy + $V=\eL(\R)$ there are no infinite $\mathcal J$-mad families. These results apply in particular when $\mathcal J$ is the ideal of finite sets $\fin$, which corresponds to the classical notion of MAD families. The proofs combine ideas from invariant descriptive set theory and forcing.
\end{abstract}

\maketitle



\section{Introduction}\label{s.intro}

\textbf{(A)} Let $\fin$ denote the ideal of finite subsets of $\omega$. Classically, a family $\mathcal A \subseteq \powerset(\omega)$ is called an \emph{almost disjoint family} (short: \emph{AD} \emph{family}) if the family $\mathcal A$ consists of infinite subsets of $\omega$ and any two distinct $x,y\in\mathcal A$ have finite intersection, that is $x\cap y\in \fin$. A \emph{maximal almost disjoint} (short: \emph{MAD}) \emph{family} is an AD family which is maximal (with respect to $\subseteq$) among AD families. Finite MAD families exist trivially, and using Choice (e.g., Zorn's Lemma), it is routine to show that there are infinite MAD families. \emph{Below we always assume, to avoid trivialities, that MAD families in question are infinite}.

\medskip

The study of the definability of (infinite) MAD families has a long history, but the area has in recent years seen a remarkable blossoming with many new results. The fundamental results in the area go back to Mathias' famous paper \cite{mathias}, where it was shown that there are no analytic MAD families, and further proved that assuming there is a Mahlo cardinal, $\ZF$+``there are no MAD families'' is consistent. Much more recently, T\"ornquist \cite{toernquist} showed that there are no MAD families in Solovay's model (from~\cite{solovay}), thus weakening the large cardinal assumption. Horowitz and Shelah \cite{horowitz-shelah} then removed the large cardinal assumption completely, showing that one can construct a model of $\ZF$ without MAD families just assuming $\ZF$ is consistent.

\medskip

The leading idea of the present paper is to use ideas from \emph{invariant descriptive set theory}, i.e. the descriptive set theory of definable equivalence relations and invariance properties, in combination with Mathias-like forcings and absoluteness to give uniform proofs of the non-existence of (definable) MAD families in various settings. This approach in turn allows us to prove vastly more general results about the definability of \emph{$\mathcal J$-MAD families}, where the ideal $\fin$ is replaced by an ideal $\mathcal J$ on $\omega$ for $\mathcal J$ in a large class of Borel ideals. However, we first give transparent and surprisingly uniform proofs of the following:

\begin{theorem}\label{t.intro.1D}
\ 

\begin{enumerate}
\item There are no analytic MAD families.
\item Under $\ZF$ + Dependent Choice + Projective Determinacy, there are no projective infinite MAD families.
\item Under $\ZF$ + Determinacy + $V=\eL(\R)$ there are no infinite MAD families.
\end{enumerate}

\end{theorem}
The first result is originally due to Mathias \cite{mathias}.
The remaining two results were also shown independently by Neeman and Norwood in \cite{neeman-norwood} using somewhat different methods.

We mention that two of the present authors have very recently found a proof (unpublished) that if every (resp., every projective) set is completely Ramsey and Ellentuck-comeager uniformization (resp., for projective relations) holds, then there are no (projective) MAD families.

\medskip

\textbf{(B)} 
We now discuss the notion of $\mathcal J$-MAD families and the generalization of Theorem \ref{t.intro.1D} to this setting. 

Given an ideal $\mathcal J$ on a countable set $S$, write $\mathcal J^+$ for $\powerset(S)\setminus \mathcal J$. 
One may define \emph{$\mathcal J$-almost disjoint} (short \emph{$\mathcal{J}$-AD}) sets in $\mathcal J^+$ and  \emph{$\mathcal J$-maximal almost disjoint} (short \emph{$\mathcal{J}$-MAD}) \emph{families} as subsets of $\powerset(\mathcal J^+)$ in the obvious manner (see Section~\ref{s.notation}).
MAD families are of course the special case $\mathcal J=\fin$, where $\fin$ denotes the ideal of finite sets (i.e., the Fr\'echet ideal).

\medskip

The motivating question for the results we now describe is:

\begin{question}
For which Borel ideals $\mathcal J$ on $\omega$ can we generalize results about MAD families to the case of $\mathcal J$-MAD families?
\end{question}

\medskip

As a first step we consider $F_\sigma$ ideals. It is well known that every $F_\sigma$ ideal $\mathcal J$ is given as the finite part of a lower semi-continuous (short: lsc) submeasure $\phi\colon \powerset(\omega) \to [0,\infty]$ as follows:
\[
\mathcal J = \fin(\phi) = \{J\subseteq \omega\mid\phi(J)<\omega\}.
\]
(See Section~\ref{s.notation} for a complete definition of these notions and \cite{mazur}[1.2] for a proof of the claim.) We will see in Section~\ref{s.1D} that Theorem~\ref{t.intro.1D} generalizes to $\mathcal J$-MAD families where $\mathcal J$ is any $F_\sigma$ ideal.

\medskip

We can reach Borel ideals that are more complex, in the sense of belonging to higher parts of the Borel hierarchy, using \emph{Fubini sums and products}.
If for each $k\in \omega$ we are given an ideal $\mathcal I_k$ on a countable set $S_k$, and an ideal $\mathcal I$ on $\omega$, 
we form an ideal $\bigoplus_{\mathcal I} \mathcal I_k$ on $S=\bigsqcup_{k\in\omega} S_k$, called the \emph{Fubini sum of $(\I_k)_{k\in\omega}$ over $\I$} as follows:
\[
\bigoplus_{\mathcal I} \mathcal I_k = \{I\subseteq S \mid \{k\in\omega\mid I \cap S_k \notin \mathcal \I_k\}\in\mathcal I\}
\]
In Section~\ref{s.2D} we show that our methods apply in this generalized setting:
\begin{theorem}\label{t.intro.2D}
Let $\mathcal J =\bigoplus_{\mathcal I} \mathcal I_k$ where $\mathcal I$ and $\mathcal I_k$ for each $k\in\omega$ are $F_\sigma$ ideals on $\omega$.
\begin{enumerate}
\item\label{t.intro.2D.analytic} There are no analytic infinite $\mathcal J$-MAD families.
\item\label{t.intro.2D.PD} Under ZF + Dependent Choice + Projective Determinacy, there are no projective infinite $\mathcal J$-MAD families.
\item\label{t.intro.2D.AD} Under ZF + Determinacy + $V=\eL(\R)$ there are no infinite $\mathcal J$-MAD families.
\end{enumerate} 
\end{theorem}

We note in passing that part (1) of Theorem \ref{t.intro.2D} in the special case of $\mathcal J=\fin\otimes\fin$, the first iteration of the Fr\'echet ideal, in itself answers a question that seems to have belonged to the folklore of the field for a long time. Here $\fin\otimes\fin$ is the Fubini sum of $\fin$ over $\fin$, that is, it consists of those $X \subseteq \omega\times\omega$ such that 
\[
\{n\in\omega\mid\{m\in\omega\mid (n,m)\in X\}\text{ is infinite}\}\text{ is finite}.
\]

\medskip

Yet more complex ideals are obtained by iterating Fubini products into the transfinite. Namely, 
given $\alpha<\omega_1$ and a sequence $\vec\phi$ of lsc submeasures of length $\alpha$, we will define in a natural manner
an ideal $\fin(\vec \phi)$ on a countable set $S$, by recursively applying Fubini products. (See Section~\ref{s.alphaD} for the detailed definition.)

One can do this in such a way that $\fin(\vec \phi)$ is $\Sigma^0_{\alpha+1}$ but not $\Sigma^0_\alpha$. A particular instance of this construction is the \emph{iterated Fr\'echet ideal} $\fin^\alpha$, obtained by iterating the Fubini product construction applied to the ideal $\fin$ transfinitely, and thus obtaining ``higher dimensional'' analogues $\fin^\alpha$ of $\fin\otimes\fin$.

\medskip

We shall see in Section~\ref{s.alphaD} that our methods apply even to the more general class of ideals described in the previous paragraph:

\begin{theorem}\label{t.intro.alphaD}
Let $\mathcal J =\fin(\vec \phi)$ be as defined in Section~\ref{s.alphaD}. 
\begin{enumerate}
\item There are no analytic $\mathcal J$-MAD families.
\item Under ZF + Dependent Choice + Projective Determinacy, there are no projective infinite $\mathcal J$-MAD families.
\item Under ZF + Determinacy + $V=\eL(\R)$ there are no infinite $\mathcal J$-MAD families.
\end{enumerate} 
\end{theorem}

By this theorem, the ideals $\mathcal J$ for which we have the familiar pattern of non-definability of $\mathcal J$-MAD families lie cofinally in the Borel hierarchy with respect to the complexity of their definition.

\medskip

\textbf{(C)} The paper is structured as follows: 
Section~\ref{s.notation} collects definitions and facts that will be used throughout---most importantly, some facts from inner model theory which will allow us to assume that our $\mathcal J$-MAD families are $\kappa$-Suslin witnessed by a tree from  a model with small $\powerset(\powerset(\omega))$.

In Section~\ref{s.1D} we prove Theorem~\ref{t.intro.1D} for any $F_\sigma$ ideal $\mathcal J$.
We describe Mathias forcing $\M^\I$ relative to an ideal $\I$ and state a crucial fact, Main Proposition~\ref{mprop.1D}, regarding $\M^\I$ when $\I$ is generated by a $\kappa$-Suslin $\mathcal J$-MAD family.
Theorem~\ref{t.intro.1D} follows quickly from Main Proposition~\ref{mprop.1D} together with the inner model theory lemmas from Section~\ref{s.notation}.

Section~\ref{s.1D.properties} collects facts about $\M^\I$, and Section~\ref{s.1D.branch} proves Main Proposition~\ref{mprop.1D}.
This proof is based on the definition of a $\mathcal J$-invariant tree together with the purely combinatorial Branch Lemma~\ref{l.1D.branch} (stating that the projection of this tree is a singleton).

In Section~\ref{s.2D} we introduce the simple Fubini product and a 2-dimensional version $\M^{\I}_2$ of $\M^{\I}$.
Theorem~\ref{t.intro.2D} follows from Main Proposition~\ref{mprop.2D} (the analogue of \ref{mprop.1D} for $\M^{\I}_2$) by the same proof as before.
Section~\ref{s.2D.properties} collects facts about $\M^\I$ and Section~\ref{s.2D.branch} proves Main Proposition~\ref{mprop.2D};
the proof of the corresponding Branch Lemma~\ref{l.2D.branch} is much more involved.

Section~\ref{s.alphaD} is structured in the same way: We introduce Fubini products $\fin(\vec \phi)$ coming from a sequence $\vec \phi$ of lsc submeasures, an $\alpha$-dimensional version $\M^{\I}_\alpha$ of $\M^{\I}$, and state Main Proposition~\ref{mprop.alphaD} for $\M^{\I}_\alpha$, from which Theorem~\ref{t.intro.alphaD} follows immediately.
Section~\ref{s.alphaD.properties} collects facts about $\M^\I_\alpha$. 
Section~\ref{s.alphaD.branch} proves Main Proposition~\ref{mprop.alphaD} via the Branch Lemma~\ref{l.alphaD.branch}, generalizing Section~\ref{s.2D.branch}.

Finally, in section \ref{s.postscript}, we briefly discuss the general (and open) problem of characterizing precisely for which Borel ideals on $\omega$ one may hope to obtain an analogue of Theorem \ref{t.intro.alphaD}.

\medskip

\textit{Acknowledgments: We owe a great debt to Sandra M\"uller for helping us find proofs of Lemmas~\ref{l.innnermodeltheory.PD} and \ref{l.innnermodeltheory.AD}. The first author thanks for support through Asger T\"ornquist's Sapere Aude level 2 grant from the Danish Council for Independent Research. The second and third authors thank the DNRF Niels Bohr Professorship of Lars Hesselholt for support. The second author also gratefully acknowledges the generous support from the Austrian Science Fund (FWF) grant number F29999. The third author thanks the Danish Council for Independent Research for support through grant DFF-7014-00145B.}


\medskip

\section{Notation and Preliminaries}\label{s.notation}

We sometimes decorate names in the forcing language with checks and dots with the aim of helping the reader.
For notation not defined here we refer to \cite{kechris,jech,kanamori,moschovakis}

\subsection{Ideals}
Fix a countable set $S$. An \emph{ideal}  on $S$ is a family $\mathcal S \subseteq\powerset(S)$ 
satisfying
\begin{enumerate}
\item $\emptyset \in \mathcal J$;
\item If $A \in \mathcal J$, then for any subset $B \subseteq A$ we have $B \in \mathcal J$;
\item If $A \in \mathcal J$ and $B \in \mathcal J$, then $A \cup B \in \mathcal J$.
\end{enumerate}
We denote by $\fin$ the ideal of finite sets. 

Given an ideal $\mathcal J$, we write $\mathcal J^+$ to denote the co-ideal, i.e., 
\[
\mathcal J^+ = \{A \subseteq S \mid A \notin \mathcal J\}.
\]
For $A$, $B \in \powerset(S)$, we write
$$
A\subseteq^*_{\mathcal J} B \Leftrightarrow (\exists I \in \mathcal J)\; A \subseteq B\cup
I.
$$
We write $A \subseteq^* B$ for $A \subseteq^*_{\fin} B$.

\medskip

We say that a family $\A \subseteq \powerset(S)$ is $\mathcal J$-\emph{almost disjoint} (short:
\emph{$\mathcal J$-AD}) if $\A \subseteq \mathcal J^+$ and for any $A$, $B \in \A$ we have $A \cap B \in \mathcal J$.
A set
$\A\subseteq \powerset(S)$ is said to be a \emph{$\mathcal J$-MAD family} if $\A$ is a $\mathcal J$-AD family which is maximal with respect to inclusion among $\mathcal J$-AD families.

\begin{definition}
Let $\A \subseteq \powerset(S)$. 
By the \emph{ideal generated by $\A$} we mean the ideal
$\I$ on $S$ defined as follows:
$$
\I = \{I \subseteq S \mid
(\exists n \in\omega)(\exists A_0, \dots, A_{n} \in \A) \;  I\subseteq\bigcup_{i\leq n} A_i\},
$$
i.e., the smallest (under $\subseteq$) ideal on $S$ containing each set from
$\A$.
\end{definition}
Suppose $\A \subseteq \powerset(S)$ and $\mathcal J$ is an ideal on $S$. Then note the ideal generated by $\A \cup\mathcal J$ is
\begin{align*}
\{I\subseteq S \mid
     I\in \mathcal J \lor (\exists n \in\omega)(\exists A_0, \dots, A_{n} \in \A) \; I \subseteq^*_{\mathcal{J}}
       \bigcup_{i\leq n}A_i\}.
\end{align*}


We point out that if $\A$ is an infinite $\mathcal J$-AD family then $[S]^{<\omega}\subseteq \mathcal J$ and  $\mathcal J$ is proper (i.e., $S \notin \mathcal J$; otherwise there are no non-empty, let alone infinite, $\mathcal J$-AD families).
Moreover we could assume $\bigcup \mathcal A = S$ (although we shall never need this).

We point out that enlarging an ideal $\mathcal J$ by an infinite $\mathcal J$-AD family yields a proper ideal.
\begin{lemma} 
Let $S$ be arbitrary, $\mathcal J$ an ideal on $S$ and $\A \subseteq \powerset(S)$ a $\mathcal{J}$-AD
family.
If 
$\A$ is infinite, the ideal $\I$ generated by $\A \cup\mathcal J$ is proper. (The other implication
holds if $\bigcup\A = S$.)
\end{lemma}
\begin{proof}
Suppose towards a contradiction that $S \in \I$. Then there exist $A_0, \dots, A_{n} \in \A$, $J \in
\mathcal{J}$ such that $S \subseteq
\bigcup_{i\leq n}A_i \cup J$.
Since $\A$ is
$\mathcal{J}$-almost disjoint $\A=\{A_0, \hdots, A_n\}$ is finite.

For the last claim, suppose $\bigcup\A = S$; show the contrapositive. If $\A$ is finite, then $S \subseteq
\bigcup_{A\in\A}A$, and thus $S \in \I$.
\end{proof}

\medskip

A \emph{submeasure} on $\omega$ is a function $\phi\colon \powerset(\omega) \to [0,\infty]$ which
satisfies
\begin{itemize}
\item $\phi(\emptyset) = 0$;
\item $\phi(X) \leq \phi(Y)$ for $X \subseteq Y$;
\item $\phi(X \cup Y) \leq \phi(X) + \phi(Y)$ for $X,Y \in \powerset(\omega)$;
\item $\phi(\{n\}) < \infty$ for every $n\in\omega$.
\end{itemize}

We say that $\phi$ is \textit{lower
semi-continuous} (lsc) if identifying $\powerset(\omega)$ with $2^{\omega}$ carrying product topology, it is lower semi-continuous as a function $\phi \colon 2^{\omega} \to [0,\infty]$,
i.e., if $X_n \to X$ implies $\lim \inf_{n\to\infty} \phi(X_n) \geq \phi(X)$. For
submeasures, this is equivalent to saying that
$\phi(X)=\lim_{n\to\infty}\phi(X \cap n)$.

As stated already in the introduction, given a submeasure $\phi$ on $\omega$
\[
\fin(\phi)=\{X\in\powerset(\omega)\mid\phi(X)<\infty\}
\]
is an $F_\sigma$ ideal on $\omega$ and every $F_\sigma$ ideal $\mathcal J \supseteq \fin$
arises in this way \cite[1.2]{mazur}.

\subsection{Trees and Suslin sets of reals}\label{s.trees}

Let $X_0,X_1$ be a sets. 
We follow established descriptive set theoretic conventions and call a \emph{tree $T$ on $X_0\times X_1$} a subset of 
$X_0^{<\omega} \times X_1^{<\omega}$ which is closed under initial segments and such that $(t_0,t_1) \in T \Rightarrow \lh(t_0) = \lh(t_1)$ (compare \cite[2.C]{kechris}). 
Given $t=(t_0,t_1) \in T$, $\pi(t)=t_0$.
For any $s\in T$,  $T_{[s]} = \{t\in T \mid t \text{ is compatible with } s\}$. 
Of course 
\[
[T]=\{(x_0,x_1)\in
X_0^{\omega}\times X_1^{\omega}\mid (\forall n\in\omega)\; (x_0 \restriction n, x_1\res n) \in T\},
\]  
and for $w=(x_0,x_1) \in [T]$, $\pi(w)=x_0$.
Finally we write
\begin{align*}
\pi[T] = \{x_0 \in X_0^{\omega} \mid (\exists x_1 \in X_1^{\omega}) \; (x_0,x_1) \in [T]\}.
\end{align*}

Recall that $\A \subseteq 2^{\omega}$ is \emph{$\kappa$-Suslin} if and only if there exists an ordinal $\kappa$ and a tree $T$ on $2 \times \kappa$ such that
\begin{align*}
\A = \pi[T] = \{x \in 2^{\omega} \mid (\exists f \in \kappa^{\omega}) \; (x,f) \in [T]\}.
\end{align*}
The analytic subsets of $2^\omega$ are precisely the $\omega$-Suslin sets.

\medskip

If $S$ is a countable set we will also talk of $\kappa$-Suslin subsets of $\powerset(S)$.
For this purpose we shall identify $S$ with $\omega$ via some fixed bijection $h \colon \omega \to S$ as well as identify each $x \subseteq S$ with its characteristic function $\chi_x$. 

Thus, 
$\A \subseteq \powerset(S)$ is $\kappa$-Suslin if and only if there is a tree $T$ on $2\times\kappa$ such that
$\A=\{x \in \powerset(S) \mid \chi_x \circ h \in \pi[T]\}$.
We shall (sloppily and through the identifications of $S$ with $\omega$ and $\chi_x$ with $x$) also write
$\A=\pi[T]$ in such a case.

\medskip

We use both $\subseteq$ and $\sqsubseteq$ for the initial segment relation for sequences.

\begin{lemma}\label{l.absoluteness}
Let $T$ be a tree on $2\times\kappa$ and let $\mathcal J$ be a Borel ideal on a countable set $S$. Then the following properties are absolute between a ground model and its forcing extension:
\begin{enumerate}
\item \label{inBaire}``$\pi[T] \subseteq \mathcal J^+$''.
\item \label{branches_are_ad}``$\pi[T]$ is an $\mathcal J$-almost disjoint family''.
\item \label{catches} ``$y$ is $\mathcal J$-almost disjoint from every set in $\pi[T]$''.
\end{enumerate}
In the above, we mean by $\mathcal J$ the ideal obtained by interpreting the Borel \emph{definition} in the current model.
\end{lemma}
\begin{proof}
(\ref{inBaire})
Let $U$ be a tree on $2\times\omega$ such that $\mathcal J = \pi[U]$.
Consider the tree $T_{+}$ on  $2\times \kappa \times U$ defined by
\begin{align*}
T_{+} = \{ (a, s,\bar u) \in 2^{<\omega}\times\omega^{\omega}  \times U^{<\omega} \mid  &\lh(a) =\lh(s)= \lh(\bar u),\\
      &(a,s) \in T \text{ and for all $k<\lh(t)$}\\
      &\text{for all $k'<k$ }\bar u(k')\subsetneq  \bar u(k)\text{ and}\\
      &a\res k+1 \subseteq \pi(\bar u(k)) \}
\end{align*}
where we identify $a \in 2^{<\omega}$ with $\{n\mid a(n)=1\}$.
Then $\pi[T]\subseteq\mathcal J^+$ if and only if $[T_{+}] = \emptyset$, which is
absolute.

\medskip
(\ref{branches_are_ad})
By the previous item is suffices to show that ``$\forall x,y\in \pi[T]$ $x\neq y \Rightarrow x\cap y \in \mathcal J$'' is absolute.
Let $U$ be a tree on $2\times\omega$ such that $\mathcal J^+ = \pi[U]$.
Consider the tree $T_{\cap}$ on  $T\times T \times 2\times\omega$ defined by
\begin{align*}
T_{\cap} = \{ (\bar t_0,\bar t_1, a, s) \in T \times T \times 2^{<\omega}\times\omega^{<\omega} \mid &\lh(\bar t_0) = \lh(\bar t_1) = \lh(a)=\lh(s),\\
&(a,s)\in U, \pi(\bar t_0)\neq \pi(\bar t_1), \\
      &\text{and for all $k< \lh(\bar t_0)$,}\\
      &\text{for all $k'<k$, }\bar t_0(k')\subsetneq  \bar t_0(k) \in T,\\
      &\bar t_1(k')\subsetneq  t_1(k) \in T\text{, and}\\
      &a \res k+1 \subseteq \pi(\bar t_0(k))\cap\pi(\bar t_1(k))\}
\end{align*}
where we identify $s \in 2^{<\omega}$ with $\{n\mid s(n)=1\}$.
Then the statement in question holds if and only if $[T_{\cap}] = \emptyset$, which is
absolute.

\medskip
(\ref{catches}) Similarly as the previous item (left to the reader).
\end{proof}

\medskip

\subsection{Inner model theory}\label{ss.inner.model.theory}

It is well known that under $\PD$ the pointclasses  $\mathbf{\Pi}^1_{2n+1}$ and $\mathbf{\Sigma}^1_{2n+2}$ are \emph{scaled} (i.e., they have the prewellordering property---see \cite{moschovakis}, \cite[Chapter~36]{kechris}, or \cite[\S30]{kanamori} for an introduction to the theory of scales).
These scales provide us with tree representations for projective sets while at the same time, each scale can be captured by a `small' model.
For the proof of the next lemma, also recall that ${\bfdelta}^1_n$ is the supremum of the lengths of $\mathbf{\Delta}^1_n$ prewellorderings of $\omega^{\omega}$ (see, e.g., \cite[p.~423]{kanamori}).

\begin{lemma}\label{l.innnermodeltheory.PD}
Assume $\PD$. Suppose $A$ is projective. 
There exists a model $M$ of $\ZFC$ and a tree $T \in M$ on $\omega \times \kappa$ (for some ordinal $\kappa$) such that
$\pi[T]=A$ and $\powerset(\powerset(\omega))^M$ is countable in $V$.
\end{lemma}
\begin{proof}
Given a projective set $A$, suppose without loss of generality that $A$ is $\mathbf{\Pi}^1_{n}$, $n \geq 3$, and $n$ is odd.
By the scale property let $T_n$ be the tree given by a $\Pi^1_{n}$ scale on a complete $\Pi^1_{n}$ set.
Fix $a \in \omega^{\omega}$ such that $A$ is $\Sigma^1_n(a)$ and let 
$M = \eL[T_n,a]$. 

\medskip

To see that $M$ satisfies what is claimed in the lemma, let
$M_{n-1}(a)$ be the class-sized iterable model with $n-1$ Woodin cardinals obtained by iterating out the top extender of $M^\#_{n-1}(a)$ (see Definitions~1.5 and 1.6 in \cite{mueller-ea}).
By $\PD$ this model exists   
and $\powerset(\powerset(\omega))^{M_{n-1}(a)}$ is countable in $V$ (by \cite[Theorem~2.1]{mueller-ea}).
As is pointed out in \cite[p.~12--13]{steel-games-scales} (the proof is said to be implicit in \cite{steel-HOD-L-of-R-core-model}) there is an iterate $Q$ of $M_{n-1}(a)$ such that
$M = \eL[Q|{\bfdelta}^1_n]$ (for this we need that $n$ is odd and $n \geq 3$). 
Moreover $\powerset(\powerset(\omega))^M$ is the same as $\powerset(\powerset(\omega))^{Q}$, which is countable in $V$.
By the presence of $T_n$ and $a$ in $M$, it is easy to obtain $T$ such that $\pi[T]=A$.

\end{proof}
There is a version of this based on the full Axiom of Determinacy ($\AD$), which we shall also use:
\begin{lemma}\label{l.innnermodeltheory.AD}
Assume $\AD$ holds and $V=\eL(\R)$. Suppose $A$ is $\mathbf{\Sigma}^2_1$. 
There exists a model $M$ of $\ZFC$ and a tree $T \in M$ on $\omega \times \kappa$ (for some ordinal $\kappa$) such that
$\pi[T]=A$ and $\powerset(\powerset(\omega))^M$ is countable in $V$.
\end{lemma}
\begin{proof}
As we are working in $\eL(\R)$, $\mathbf{\Sigma}^2_1$ and $\Sigma_1(\R\cup\{\R\})$ are the same pointclass (see, e.g., \cite[p.~13]{steel-games-scales}).
Under the hypothesis of the lemma, by \cite{martin-steel} this pointclass is scaled; let $T^*$ be the tree coming from this scale.
According to \cite[p.~13]{steel-games-scales}, \cite{steel-HOD-L-of-R-core-model} shows that $\eL[T^*]= \eL[Q]$ where $Q$ is an iterate of an initial segment of $M_\omega$ and again it holds that
$\powerset(\powerset(\omega))^{\eL[T^*]}$ is countable in $V$.
Moreover $\mathcal A = \pi[T]$ for some tree  $T$ such that $T \in \eL[T^*]$.
\end{proof}
Finally we shall need the following result (due to Woodin) known as Solovay's Basis Theorem (see \cite[Remark~2.29(3)]{koellner-woodin}).
\begin{fact}\label{f.innnermodeltheory.reflection}
Assume $\AD$ holds and $V=\eL(\R)$. Then every $\mathbf{\Sigma}^2_1$ statement is witnessed by a set $A\subseteq \R$ which is itself $\mathbf{\Delta}^2_1$. 
\end{fact}


\medskip

\section{Classical MAD families (and a bit more)}\label{s.1D}

In this section we give proofs of the following:
\begin{theorem}\label{t.main.1D}
Let $\mathcal J =\fin$, or more generally $\mathcal J=\fin(\phi)$ where $\phi$ is an lsc submeasure on $\omega$.
\begin{enumerate}
\item There are no analytic infinite $\mathcal J$-MAD families.
\item Under ZF + Dependent Choice + Projective Determinacy, there are no projective infinite $\mathcal J$-MAD families.
\item Under ZF + Determinacy + $V=L(\R)$ there are no infinite $\mathcal J$-MAD families.
\end{enumerate} \end{theorem}
The first item was first shown by Mathias \cite{mathias} (at least in the case of $\fin$). The next two items are independently, and by a somewhat different method shown by Neeman and Norwood \cite{neeman-norwood} (also in the case of $\fin$).

\medskip

We use the following close relative of Mathias forcing:
\begin{definition} \label{mathiasforcing}
Suppose that $\I \supseteq \fin$ is an ideal on $\omega$, and $\I^+$ its co-ideal. Define
$$
\M^{\I} = \{(a,A) \mid a\in[\omega]^{<\omega}, A \in \I^+, \max(a)
< \min(A)\}
$$
ordered by
$$
(a',A') \leq (a,A) \text{ if and only if } a\sqsubseteq a'\cup A' \subseteq a \cup A.
$$
Of course for $X,Y\subseteq \omega$,  $X\sqsubseteq Y$ means $X = Y\cap (\max(X)+1)$.
We write $\M$ for $\M^{\fin}$.
\end{definition}

We use the following notation, which should be familiar enough:
\begin{notation}
\begin{enumerate}
\item
Given a filter $G$ on $\M^{\I}$, let
$$
x_G = \bigcup\{a \mid (\exists A \in\I^{+}) \; (a,A) \in G\}.
$$
Note that if $G$ is $\M^{\I}$-generic then $x_G \notin \fin$.
\item
For $(a,A) \in \M^{\I}$, and $b\subseteq A$
finite, let $A / b = \{n\in A \mid n > \max(b)\}$.
\item
For $p\in \M^{\I}$, we write $p=(a(p),
A(p))$ when we want to refer to its components.
\item
For $p\in \M^{\I}$, we let $\M^{\I}(\leq p) = \{q\in\M^{\I} \mid q \leq p\}$.
\end{enumerate}
\end{notation}

\begin{assumption}\label{assumption.1D}
Until the end of Section~\ref{s.1D} let $\mathcal J = \fin$ or more generally $\mathcal J=\fin(\phi)$ and assume $\A \subseteq \powerset(\omega)$ is an infinite $\mathcal J$-AD family
which is $\kappa$-Suslin. 
Fix a tree $T$ on $2\times
\kappa$ such that $\pi[T] = \A$. Let $\mathcal I$ be the ideal generated by $\mathcal A
\cup \mathcal J$.
\end{assumption}

To avoid overly cumbersome notation, we shall phrase our presentation in terms of the ideal $\fin$. %
However this section is written so that whenever relevant, the reader may replace $\fin$ (but \emph{not} the word ``finite'' or the expression $[\omega]^{<\omega}$) with $\fin(\phi)$, for any lsc submeasure $\phi$ on $\omega$, in which case she must also replace ``almost disjoint'' by ``$\fin(\phi)$-AD'', etc.
We will point out how to modify proofs when these trivial substitutions do not suffice.

\medskip

The main workload in the proof Theorem~\ref{t.main.1D} is carried by the following Main Proposition, of which we give a proof in Section~\ref{s.1D.branch} after we collect some properties of the forcing $\M^{\I}$ in Section~\ref{s.1D.properties}.
\begin{mprop}\label{mprop.1D}
$\forces_{\M^{\I}} (\forall y \in \pi[T]) \; y \cap x_{\dot G} \in \fin$.
In other words, $\forces_{\M^{\I}} x_{\dot G} \notin \pi[T]$ and $\{x_{\dot G}\}\cup
\pi[T]$ is an almost disjoint family.
\end{mprop}
Before we prove the Main Proposition, we show how easily it leads to Theorem~\ref{t.main.1D}.
Firstly, we give a very short proof of the
classical result that there are no analytic MAD families:
\begin{corollary}[\cite{mathias}]\label{cor.1D.analytic}
There are no analytic MAD families.
\end{corollary}
\begin{proof}
Suppose $\A$ is an analytic almost disjoint family, and fix a tree $T$ on
$2\times\omega$ such that $\A= \pi[T]$ (identifying $\powerset(\omega)$ with $2^\omega$).
By Levy-Shoenfield Absoluteness 
$\pi[T]^{V[G]}$ should be maximal in any forcing extension $V[G]$ of $V$; but by
Main Proposition~\ref{mprop.1D}, there is a forcing extension $V[G]$ containing a real which is almost disjoint from any set in $\pi[T]^{V[G]}$.
\end{proof}

We likewise obtain an easy and transparent proof that under projective determinacy, there are no projective MAD families.
Here we make use of the inner model theory facts from Section~\ref{ss.inner.model.theory}.

\begin{corollary}\label{cor.1D.PD}
Under $\PD$ there are no projective MAD families.
\end{corollary}\label{t.pd}
\begin{proof}
Assume $\PD$ holds and suppose $\mathcal A$ is an infinite almost disjoint family which is projective.
Fix a tree $T$ so that $\mathcal A=\pi[T]$ and a model $M$ as in the previous lemma.
Note that $M \vDash \pi[T]$ is an infinite almost disjoint family. 
Working inside $M$ 
let $\I$ be the ideal generated by $\fin\cup\pi[T]$ and let $\P$ denote $\M^{\I}$ in $M$. 
As $\powerset(\powerset(\omega))^M$ is countable in $V$ we may find $r \in [\omega]^\omega$ which is $\P$-generic.
By Main Proposition~\ref{mprop.1D}
\[
M[r]\vDash (\forall y \in \pi[T])\; y \text{ is almost disjoint from }r.
\]
By Item~\ref{catches} of Lemma~\ref{l.absoluteness} the statement on the right is absolute for models of ZFC and therefore holds in $V$.
Thus, $\mathcal A$ is not maximal.
\end{proof}

A similar proof can be given of the $\AD$ analogue:
\begin{corollary}\label{cor.1D.AD}
If $\eL(\R)\vDash \AD$, there are no MAD families in $\eL(\R)$. 
\end{corollary}
\begin{proof}
Suppose towards a contradiction that $V=\eL(\R)$, $\AD$ holds, and there is a MAD family.
As the existence of a MAD family is a $\Sigma^2_1$ statement, 
by Lemma~\ref{f.innnermodeltheory.reflection}, there is a $\Sigma^2_1$ MAD family $\mathcal A$.
By Lemma~\ref{l.innnermodeltheory.AD} we may pick an ordinal $\kappa$ and a tree $T$ on $\kappa\times\omega$ such that $\pi[T]= \mathcal A$.
Moreover, there is a model $M$ such that $T\in M$ and $\powerset(\powerset(\omega))^{M}$ is countable.
Now argue precisely as in Corollary~\ref{cor.1D.PD} above to show that 
$\mathcal A$ is not maximal, reaching a contradiction.
\end{proof}

\medskip

\subsection{Properties of Mathias forcing relative to an ideal}\label{s.1D.properties}

For the proof of the Main Proposition \ref{mprop.1D} in the next section, 
we need to explore the
immediate properties of the forcing notion $\M^{\I}$. 

\medskip

The following lemma holds for any ideal $\mathcal I\supseteq\fin$ and under our Assumption~\ref{assumption.1D}.
\begin{lemma}~\label{l.1D.properties}
\begin{enumerate}
\item \label{xGisAD}$\forces_{\M^{\I}} (\forall y \in \I)\; x_{\dot G} \cap y \in \fin$.
\item\label{1D.large.x_G} $\forces_{\M^{\I}} x_{\dot G} \in \fin^+$.
\item \label{isomorphicPO} Fix $A\in\I^+$ and $a_0, a_1 \in [\omega]^{<\omega}$ with $\max
(a_i) < \min (A)$ for each $i \in \{0,1\}$. Let $p_i = (a_i, A)$. Then $h\colon \M^{\I}(\leq
p_0) \to \M^{\I}(\leq p_1)$ given by
$$
h({a_0}\cup b,B) = ({a_1}\cup b, B),
$$
where $b \subseteq A$ is finite and $B \subseteq A / b$, is an isomorphism of partial orders.

\item \label{finitepartnegligible} For $p_0, p_1$ as above, $\theta$ a formula in the
language of set theory, and $v \in V$ it holds that
$$
p_0 \forces \theta(v, [x_{\dot G}]_{E_0}) \text{ if and only if }
p_1\forces \theta(v, [x_{\dot G}]_{E_0})
$$

\end{enumerate}
\label{factsaboutMI}
\end{lemma}
\begin{proof}
(\ref{xGisAD}) For any $y \in \I$, the set
\begin{align*}
D_y = \{p\in\M^{\I} \mid  A(p) \cap y = \emptyset\}
\end{align*}
is dense in $\M^{\I}$, which implies that for any generic $G$ we have $x_G \cap y \in \fin$.

\medskip
(\ref{1D.large.x_G}) We verify the general case where $\mathcal J=\fin(\phi)$.
Supposing $p\forces x_{\dot G}\in\fin(\phi)$ we can find $p' \leq p$ and $n\in\omega$ so that
$p' \forces \phi(x_{\dot G})< \check n$. Since $\phi$ is lower semi-continuous and $\phi(A(p'))=\infty$ we can find a finite set $a$ such that $a(p') \sqsubseteq a \subseteq A(p')$ and $\phi(a) > n$.
Since $(a,A(p')/a)\forces a\subseteq x_{\dot G}$ we reach a contradiction.

\medskip
(\ref{isomorphicPO}) Immediate from the definitions.

\medskip
(\ref{finitepartnegligible}) Suppose $p_1 \forces \theta(v, [x_{\dot
G}]_{E_0})$. Let $G$ be a generic such that $p_0 \in G$.  Use $h\colon \M^{\I}(\leq p_0) \to
\M^{\I}(\leq
p_1)$ from (\ref{isomorphicPO}) to obtain a generic $h(G)$ containing $p_1$. Since $x_G
E_0 x_{h(G)}$, we conclude $\theta(v, [x_{\dot G}]_{E_0})$, proving that ``if''
holds. The proof of ``only if'' is analogous.
\end{proof}

Furthermore, we have the following diagonalization result.

\begin{lemma}\label{diagonalization.1D}
Let $(A_k)_{k\in\omega}$ be a sequence from $\I^+$ 
satisfying that $A_{k+1} \subseteq A_k$ for every $k\in\omega$. Then there is
$A_{\infty} \in \I^{+}$ such that $A_{\infty} \subseteq^* A_k$ for every $k\in\omega$.
\end{lemma}
In both the lemma and its proof for the case $\mathcal J=\fin(\phi)$ there is no need to replace $\subseteq^*$ by
$\subseteq^*_{\fin(\phi)}$.

Also note that it follows that the preorder $( \I^{+},
\subseteq^*_{\mathcal J})$ is $\sigma$-closed.
In fact, $\I^{+}$ is a selective co-ideal---however, we will only need the statement in the
lemma.

\begin{proof}
We construct two sequences $( B_n)_{n\in \alpha}$ and  $(C_n)
_{n\in \alpha}$ of length $\alpha \leq \omega$ such that for each $n < \alpha$, 
\begin{itemize}
\item $B_n \subseteq A_n$;
\item for each $m$, $B_n \subseteq^*_{\I} A_m$;
\item $C_n \in \A\setminus\{A_i\mid i<n\}$;
\item $B_n \cap C_n \notin \fin$ and $B_n \cap C_i \in \fin$ for $i<n$.
\end{itemize}
Suppose we have found $B_i$ and $C_i$ as above for $i<n$. 
Define a sequence $m_0, m_1, \hdots$ from $\omega$ by recursion on $k$ as follows:
\[
m_k = \min\Big( A_{n+k} \setminus \big(\{ m_i  \mid i < k\} \cup
          \bigcup_{i < n} C_i\big)\Big)
\]
and let $B = \{ m_k \mid k\in \omega\}$.

In the case of $\fin(\phi)$, instead chose finite sets $M_0, M_1, \hdots$ such that $M_k\subseteq A_{n+k}\setminus\big(\bigcup_{i < n} C_i\cup M_i\big)$
and $\phi(M_k)>0$ for each $k\in\omega$.
This is possible since for each $k$, $A_{n+k} \setminus 
          \big(\bigcup_{i < n} C_i \cup M_i\big)  \in \fin(\phi)^+$.
Then let $B=\bigcup_{k\in\omega}M_k$.

\medskip

If $B \in\I^+$, we let $\alpha = n-1$ and we are done, since $B \subseteq^* A_i$ for
every $i \in \omega$.
If on the other hand $B \notin \I^+$, we let $B_n = B$; since $B \in \fin^+$ we can pick
$C_n \in \A\setminus\{C_i\mid i<n\}$ such that $B_n \cap C_n \notin \fin$.

Supposing that the construction does not end at a finite stage, let $A_{\infty} =
\bigcup_{n\in\omega}B_n\cap C_n$.  It
is clear by construction that $A_{\infty} \subseteq^* A_m$ for every $m\in\omega$.
Furthermore, since $A_{\infty}$ is an infinite union of sets not in $\fin$ which are also
subsets of distinct elements in $\A$, and the latter is an almost disjoint family, we conclude that
$A_{\infty} \in \I^+$.
\end{proof}

\begin{lemma}\label{HVDinV}
Let $HVD(X)$ denote the sets which are hereditary definable using parameters from $V
\cup \{X\}$. Then the following holds:
$$
\forces_{\M^{\I}}(\ON^{\omega})^{HVD([x_{\dot G}]_{E_0})} \subseteq V.
$$
\end{lemma}
\begin{proof}
Suppose $\theta(x_1, x_2, x_3, x_4)$ is a formula with all free variables shown,
$p_0\in \M^{\I}$, $a$ is arbitrary, and $\dot x$ is a $\M^{\I}$-name
such that
\begin{align*}
p_0 \forces \dot x\colon\omega \to \ON \land
            (\forall n\in\omega) (\forall \alpha\in\ON)\;
            \dot x(n) = \alpha \Leftrightarrow \theta(n,\alpha,\check a,[x_{\dot G}]_{E_0}).
\end{align*}
Let $A_0 = A(p_0)$, and build $A_0 \supseteq A_1 \supseteq A_2 \supseteq \cdots$ and
$\alpha_0, \alpha_1, \alpha_2, \dots$ a sequence of ordinals as follows: given $A_n$, find
$(b,A_{n+1}) \leq (a(p_0),A_n)$ and $\alpha_{n}$ such that
\begin{align*}
(b,A_{n+1}) \forces \theta(n,\check \alpha_{n},\check a,[x_{\dot G}]_{E_0}).
\end{align*}
Finally, find $A_{\infty}$ such that $A_{\infty} \subseteq^* A_n$ for every $n\in\omega$.

W claim that $(a(p_0),A_{\infty}) \forces (\forall n\in \omega)\; \dot x(n) = \check \alpha_n$, and thus
$\dot x \in V$.
To prove this, suppose towards a contradiction that there is $n\in \omega$ such that
$(a(p_0),A_{\infty}) \notforces\dot x(n) = \check \alpha_n$, and find $(b,B) \leq
(a(p_0),A_{\infty})$ such
that $(b,B) \forces \dot x(n) \neq \check \alpha_n$. That is, $(b,B) \forces \neg
\theta(n,\check \alpha_n,\check a,[x_{\dot G}]_{E_0})$. By 
Fact~\ref{factsaboutMI}(\ref{finitepartnegligible}), also $(a(p_0),B)\forces
\neg\theta(n,\check \alpha_n,\check a,[x_{\dot G}]_{E_0})$. However, since $B \subseteq
A_{\infty} \subseteq^* A_{n+1}$ we know that 
$(a(p_0), B \cap A_{n+1}) \leq (a(p_0),B), (a(p_0), A_{n+1})$.
This contradicts
the fact that $(a(p_0),A_{n+1}) \forces \theta(n,\check \alpha_n,\check a, [x_{\dot G}]_{E_0})$.

\end{proof}

\medskip

\subsection{The Branch Lemma}\label{s.1D.branch}

In this section we shall finally prove the Main Proposition~\ref{mprop.1D}.
We make a crucial definition (imported from \cite{toernquist}), followed by some fairly straightforward observations:

\begin{definition}\label{d.1D.T^x}
For $x \subseteq \omega$, let
\begin{align*}
T^x = \{t\in T \mid (\exists w \in [T_{[t]}]) \; \pi(w)\cap x \notin\fin\}.
\end{align*}
\end{definition}

\begin{facts}~ \label{factsaboutTx}
\begin{enumerate}

\item If $x \mathbin{E_0} z$, then $T^x = T^z$. This means that for a generic $G$, the
tree $T^{x_{G}}$ is definable from $[x_G]_{E_0}$.

\item\label{factsaboutTx.start} $T^x$ is a pruned tree on $2\times\kappa$.

\item $t\in T^x$ if and only if there is some $y\in \pi[T_{[t]}^x]$ such that $y\cap x
\notin\fin$.

\item $\emptyset \notin T^x$ is equivalent to $T^x = \emptyset$, as well as to $[T^x] =
\emptyset$, as well as to that $\{x\}\cup \A$ is not an AD family.

\item\label{factsaboutTx.end} Since $T^x$ is a subtree of $T$, $\pi[T^x]\subseteq \A$.
\end{enumerate}
\end{facts}

The proof of the Main Proposition is actually based on the following Branch Lemma.

\begin{blemma}\label{l.1D.branch}
$\forces_{\M^{\I}} |\pi[T^{x_{\dot G}}]| \leq 1$.  \end{blemma}

Momentarily assuming the Branch lemma,  we can very quickly show the Main Proposition~\ref{mprop.1D}, i.e., that
$$\forces_{\M^{\I}} (\forall y \in \pi[T]) \; y \cap x_{\dot G} \in \fin$$
as follows.
\begin{proof}[Proof of the Main Proposition~\ref{mprop.1D}]
Towards a contradiction, suppose $G$ is $\M^{\I}$-generic and we have $y \in {\pi[T]}^{V[G]}$ such that $y
\cap x_G \notin \fin$. 
By the Branch Lemma $\pi[T^{x_G}] = \{y\}$. 
Thus, since $y$ is definable from $[x_{\dot G}]_{E_0}$, we have $y
\in \pi[T] \cap V \subseteq \I$ by Lemma \ref{HVDinV}. 
But then by \ref{factsaboutMI}(\ref{xGisAD}), $x_G \cap y \in\fin$, contradiction.  \renewcommand{\qedsymbol}{{\tiny Main Proposition \ref{mprop.1D}.} $\Box$}
\end{proof}

\medskip

For the proof of Theorem~\ref{t.main.1D}, it remains but to prove the Branch Lemma.
\begin{proof}[Proof of the Branch Lemma~\ref{l.1D.branch}.]
Towards a contradiction, suppose $G$ is $\M^{\I}$-generic  and we have  
distinct $x_0, x_1 \in \pi[T^{x_G}]$. Fix $n$ such that $x_0\restriction{n} \neq
x_1\restriction{n}$, and let $s_i = w_i\restriction{n}$ where $x_i = \pi(w_i)$ and
$w_i \in [T]$.

\begin{claim}\label{disjointExtension}
There exists $t_0, t_1 \in T^{x_G}$ such that
\begin{enumerate}
\item $s_i \sqsubseteq t_i$ for $i\in\{0,1\}$;
\item $\forall x_0^*, x_1^*$ such that $x_i^*\in \pi[T_{[t_i]}^{x_G}] $ it holds that $x_0^*
\cap x_1^* \subseteq \pi(t_0) \cap \pi(t_1)$.
\end{enumerate}
\end{claim}

\begin{proof}[Proof of claim.]
Suppose otherwise. Then for all $t_0, t_1 \in T^{x_G}$ extending $s_0,s_1$ respectively,
there exists $x_0^*, x_1^*$ such that $x_i^* \in \pi[T^{x_G}_{[t_i]}]$ and $\pi(t_0) \cap \pi(t_1)
\subsetneq x_0^* \cap x_1^*$. We may build branches $w_0^*,
w_1^*\in [T^{x_G}]$ such that $s_i \sqsubseteq w_i^*$ and
$\pi(w_0^*)\cap\pi(w_1^*)\notin \fin$. This however contradicts the fact that
$\pi[T^{x_G}]\subseteq \pi[T]$, which is an almost disjoint family.
\renewcommand{\qedsymbol}{{\tiny Claim \ref{disjointExtension}.} $\Box$}
\end{proof}

Thus, pick $t_0, t_1 \in T^{x_G}$ as in the claim, and let
\begin{align*}
y_i = \bigcup \pi[T^{x_G}_{[t_i]}], \quad i\in\{0,1\}.
\end{align*}
It must be the case that $y_0  \in V$ since $y_0$ is definable from $[x_G]_{E_0}$ (the same is true of $y_1$).
Noting $y_0  \in \fin^+$, one of the two following cases occurs:

\medskip

\textbf{Case 1:} $x_G \cap y_0 \in \fin$. 
This, however, is a contradiction; indeed, since $y_0 = \bigcup \pi[T^{x_{\dot G}}_{[t_0]}]$ where $t_0 \in
T^{x_{\dot G}}$, Facts \ref{factsaboutTx} yields the existence of a set $y \in
\pi[T^{x_{\dot G}}_{[t_0]}]$ such that $y \cap x_{\dot G} \notin \fin$.

\medskip

\textbf{Case 2:} If the first case fails, since $\{p\in \M^\I\mid A(p) \subseteq^* y_0 \vee A(p) \cap y_0 \in \fin\}$ is dense in $\M^\I$ 
we have $x_G \subseteq^* y_0$. But then $x_G \cap y_1 \in \fin$. 
This is also a contradiction, for the same reasons as above.
\renewcommand{\qedsymbol}{{\tiny Lemma \ref{l.1D.branch}.} $\Box$}
\end{proof}


\medskip

\section{Simple Fubini products}\label{s.2D}

The ideas from the previous section can be used to prove similar results about ideals
that are further up the Borel hierarchy. In this section, we will take one step up the
ladder, whilst in the following section we see that we can go all the way. 

\medskip


Recall from Section~\ref{s.intro} that
given ideals  $\mathcal J_*$, $\mathcal J_k$   on $\omega$ (for each $k\in \omega$)
we can form the ideal $\bigoplus_{\mathcal J_*} \mathcal J_k$ on $\omega\times\omega$. 
If $\mathcal J_k = \mathcal J'$ for each $k\in\omega$, one writes $\mathcal J_* \otimes \mathcal J'$ for $\bigoplus_{\mathcal J_*} \mathcal J_k$ (called the \emph{Fubini product of $\mathcal J_*$ with $\mathcal J'$}).

\medskip
We will study ideals of the form $\mathcal J=\bigoplus_{\fin(\phi)}\fin(\phi_k)$, where $\phi$ and $\phi_k$  for each $k\in\omega$ are lsc submeasures on $\omega$. 
Clearly this includes $\fin\otimes\fin$, which is
$\fin(\phi)\otimes\fin(\phi)$ where $\phi$ is the counting measure.
For $X \subseteq \omega\times\omega$ we write 
\begin{align*}
X(n) &= \{k \in \omega\mid (n,k)\in X\},\\
\dom^{\mathcal J}_{\infty}(X) &= \{n\in \omega \mid X(n) \notin
\mathcal \fin(\phi_n)\}.
\end{align*}
We write $\dom_\infty$ for $\dom^{\fin\otimes\fin}_\infty$, and note that
$$\mathcal \fin\otimes\mathcal \fin = \{X\subseteq \omega \times \omega \mid \dom_\infty(X)\in\fin \}.$$

\medskip

We will use the two following orderings on $\powerset(\omega\times\omega)$.
For
$X \subseteq\omega\times\omega$ finite and $Y \subseteq \omega\times\omega$ we say
\begin{align*}
X \initseg_2 Y \Leftrightarrow \dom(X) \initseg \dom(Y) \land (\forall n\in\dom(X)) \; X(n)
\initseg Y(n),
\end{align*}
and
\begin{align*}
X \sqsubset_2 Y &\Leftrightarrow \dom(X) \sqsubsetneq \dom(Y) \land (\forall n\in \dom(X)) \;
X(n) \sqsubsetneq Y(n)
\intertext{(of course $X\sqsubseteq Y\iff X = Y\cap (\max(X)+1)$). In the general case where $\mathcal J =  \bigoplus_{\fin(\phi)}\fin(\phi_k)$, let} 
X \sqsubset_2 Y &\Leftrightarrow 
X\sqsubseteq_2 Y \land
 \phi(\dom(X)) <  \phi(\dom(Y)) \land \\
&(\forall n\in \dom(X)) \;\phi_n(X(n))<\phi_n(Y(n)).
\end{align*}
This section was written so that most proofs generalize almost mechanically from $\fin\otimes\fin$ to the above more general case; often this is made possible by the definition of $\sqsubset_2$ given above.

\medskip
We let as usual $(\fin\otimes\fin)^+$ (resp., $\mathcal J^+$) denote the co-ideal.

\begin{definition}
Let $(\fin\otimes\fin)^{++}$ denote the set of $A \in (\fin\otimes\fin)^+$ such that
for all $k \in \dom(A)$, $A(k)\notin \fin$.

Conditions of  the forcing notion $\M_2$ are pairs $(a,A)$ where 
\begin{enumerate}[label=(\alph*)]
\item $a\subseteq\omega\times\omega$ and is finite;
\item\label{r.forcing.large} $A \in (\fin\otimes\fin)^{++}$;
\item $\max(a(k)) < \min(A(k))$ for every $k\in\dom(a)$;
\item $\dom(a) \initseg \dom(A)$.
\end{enumerate}
We let $(a',A')\leq (a,A)$ just in case $A'\subseteq A$, and $a
\sqsubseteq_2 a' \subseteq a \cup A$.

\medskip
In the general case when $\mathcal J =  \bigoplus_{\fin(\phi)}\fin(\phi_k)$, 
$\mathcal J^{++}$ denotes the set of $A \in \mathcal J^+$ such that
for all $k \in \dom(A)$, $A(k)\notin \fin(\phi_k)$.
Moreover, replace \ref{r.forcing.large} in the definition\footnote{This makes the designation somewhat ambiguous; i.e., $\M_2$ depends on the ideal $\mathcal J$ being considered. Note that the $\M^\I_2$ notation does not provide a way to refer to these variants, unlike in the previous section.} of $\M_2$ by $A \in \mathcal J^{++}$.
\end{definition}

\medskip

Note that if $(a,A)$ is a condition in $\M_2$ then for every $k\in\dom(a)$,
the pair $(a(k), A(k))$ is a Mathias forcing condition (resp., a condition in $\M^{\fin(\phi_k)}$). Moreover, the pair $(\dom(a), \dom(A))$
is a Mathias forcing condition (resp., a condition in $\M^{\fin(\phi)}$) as well. 

\medskip

As in the $1$-dimensional case, a
relativized forcing notion is needed.

\begin{definition}
If $\I^+$ is a co-ideal of an ideal $\I\supseteq\fin\otimes\fin$,
then we write $\I^{++}$ for $\I^{+} \cap (\fin\otimes\fin)^{++}$. We let
$$
\M_2^{\I}=\{(a,A)\in \M_2: A\in\I^{++}\}
$$
equipped with the ordering inherited from $\M_2$.
\end{definition}

Note that if $\I=\mathcal J$ then $\M_2^{\mathcal I}=\M_2$. Note furthermore that if $A \in
\I^+$, then we can always find a subset $B \subseteq A$ such that $B \in \I^{++}$. We need
to establish some notation:

\medskip
\begin{notation}~
\begin{enumerate}
\item
Given a filter $G$ on $\M_2^{\I}$, let
$$
x_G=\bigcup\{a:(\exists A) (a,A)\in G\}.
$$
It is easy to see $\forces x_{\dot G}\in (\mathcal J)^{++}$ when $\mathcal J=\fin\otimes \fin$; we will check this more carefully for $\bigoplus_{\fin(\phi)}\fin(\phi_k)$ below, see Lemma~\ref{l.2D.properties}(\ref{2D.large.x_G}).

%
\item
For $p\in\M^{\I}_2$, we write $p = (a(p), A(p))$ when we want to refer to the
components of $p$.
\item
For $(a,A) \in \M_2^{\I}$ and $b\subseteq a\cup A$ finite, let
\begin{align*}
A / b =
\bigcup_{n\in N} A(n) \setminus \{m\in\omega \mid m\leq \max(b(n))\}
\end{align*}
where $N=\dom(b) \cup [\max(\dom(b))+1,\infty)$.
\item For $p\in\M^{\I}_2$, let $\M^{\I}_{2}(\leq p) = \{q \in \M^{\I}_2 \mid q \leq p\}$.
\end{enumerate}
\end{notation}


\medskip
\begin{remark}
Note that in order to meaningfully talk about $\kappa$-Suslin sets in $\powerset(\omega\times\omega)$, we identify $\omega\times\omega$ with $\omega$ (via some fixed bijection), sets with their characteristic functions, and in effect, $\powerset(\omega\times\omega)$ with $2^\omega$ (see also Section~\ref{s.trees}).
\end{remark}

\begin{assumption}
Until the end of Section~\ref{s.2D} let $\mathcal J = \fin\otimes\fin$, or more generally let $\mathcal J = \bigoplus_{\fin(\phi)}\fin(\phi_k)$ as above. 
Moreover suppose $\A \subseteq \powerset(\omega\times\omega)$ to be a
$\mathcal J$-almost disjoint family which is $\kappa$-Suslin and fix a tree $T$ on $2\times\kappa$  
such that $\pi[T] = \A$. Finally, let $\I$ be the ideal generated by $\A \cup \mathcal J$.
\end{assumption}

To ease the notation, we will focus our attention on $\mathcal J=\fin\otimes\fin$. 
However, our proofs work for $\mathcal J=\bigoplus_{\fin(\phi)}\fin(\phi_k)$ as above.
For the general case, substitute $\fin\otimes\fin$ (but \emph{not} the word finite or the expression $[\omega^2]^{<\omega}$) by $ \bigoplus_{\fin(\phi)}\fin(\phi_k)$, $\dom_\infty$ by $\dom^{\mathcal J}_\infty$, etc.~wherever relevant, unless we provide commentary.

\medskip

Now we are ready to state the Main Proposition regarding $\M^{\I}_2$ from which  Theorem~\ref{t.intro.2D} follows as a corollary, precisely analogous to the previous section. 
The proof of the Main Proposition will again rely on a Branch Lemma and will be postponed until Section~\ref{s.2D.branch}.

\begin{mprop}\label{mprop.2D}
$\forces_{\M^{\I}_2} (\forall y \in \pi[T]) \; y \cap x_{\dot G} \in \fin\otimes\fin$.
\end{mprop}
As in the one-dimensional case, our main result about $\fin\otimes\fin$ also follows directly from the Main Proposition.
\begin{corollary}\label{cor.2D}
Theorem~\ref{t.intro.2D} holds.
\end{corollary}
\begin{proof}
The proofs are essentially identical to those of Corollary~\ref{cor.1D.analytic}, Corollary~\ref{cor.1D.PD}, and Corollary~\ref{cor.1D.AD}, simply substituting $\M^{\I}_2$ for $\M^{\I}$.
\end{proof}

\medskip

\subsection{Properties of the two-dimensional forcing}\label{s.2D.properties}

Before we can prove the Main Proposition, we shall collect some of the necessary facts  about the forcing $\M^{\I}_2$.

\begin{lemma}\label{l.2D.properties}~
\begin{enumerate}
\item\label{l.2D.properties.I} For any $A \in \mathcal I$, $\forces_{\M^{\I}_2} x_{\dot G} \cap \check A \in \fin\otimes\fin$.
\item\label{decomposition.2D}
For any $k \in \omega$ the partial order $\M_2^{\I}$ is isomorphic to the product
$\M^k \times \M_2^{\I}(\leq(\emptyset, (\omega\setminus k)\times\omega))$, where $\M^k$ is the set of $k$-tuples of
classical (1-dimensional) Mathias forcing conditions. 
In the general case where $\mathcal J = \bigoplus_{\fin(\phi)}\fin(\phi_i)$
we have
\[
\M^\I_2 \cong \Big(\prod_{i<k} \M^{\fin(\phi_i)} \Big) \times \M^{\I}_2\Big(\leq\big(\emptyset, (\omega\setminus k)\times\omega\big)\Big).
\]
\item\label{2D.large.x_G} $\forces_{\M^{\I}_2} x_{\dot G} \in (\fin\otimes\fin)^{++}$.
\end{enumerate}
\end{lemma}
\begin{proof}
(\ref{l.2D.properties.I}) By an easy density argument.

\medskip
(\ref{decomposition.2D}) Define a map $\phi\colon \M^k \times \M_2^{\I}  \to \M_2^{\I}$ by
\begin{align*}
((c_i,C_i)_{i < k},(a,A)) \mapsto
(\bigcup_{i < k} \{i\}\times c_i \cup a, \bigcup_{i < k}\{i\}\times C_i \cup  A)
\end{align*}
This map is easily seen to be bijective and order preserving. The same definition works in the general case.

\medskip
(\ref{2D.large.x_G}) Work in the general case where $\mathcal J = \bigoplus_{\fin(\phi)}\fin(\phi_i)$.
First note that $\forces_{\M^\I_2} \dom(x_{\dot G})=\dom_\infty(x_{\dot G})$:
For let $n$ and $p$ be such that $p\forces \check n\in \dom(x_{\dot G})$. It must hold that $n\in\dom(a(p))$.
By Lemma~\ref{l.1D.properties}(\ref{1D.large.x_G}), 
\[
\big(a(p)(n), A(p)(n)\big)\forces_{\M} x_{\dot G} \notin \fin(\phi_n)
\] 
so by item (\ref{decomposition.2D}) of the present lemma, $p\forces_{\M^\I_2}  \check n\in \dom_\infty(x_{\dot G})$.

It remains to show $\forces_{\M^\I_2}  \dom_\infty(x_{\dot G}) \notin \fin(\phi)$.
Towards a contradiction, suppose there is $n$ and $p$ so that $p\forces \phi(\dom(x_{\dot G})) < \check n$.
Find a finite set $d$ such that $\dom(a(p)) \sqsubseteq d\subseteq \dom(a(p))\cup\dom(A(p))$ and $\phi(d) >n$, and $a$ such that
$a(p) \sqsubseteq_2 a \subseteq a(p) \cup A(p)$ and $\dom(a)=d$.
We reach a contradiction since $(d,A(p)/d) \forces d\subseteq \dom(x_{\dot G})$ and $\phi(d) > n$. 
\end{proof}

We prove a general diagonalization result (which shall be put to use in Lemma~\ref{l.force.finite.bit} below):

\begin{lemma}\label{diagonalizing.A}
Let $(A_k)_{k\in\omega}$ be a sequence from $\I^{++}$ 
satisfying that $A_{k+1} \subseteq A_k$ for every $k\in\omega$. Then there is
$A_{\infty} \in \I^{++}$
such that $A_{\infty} \subseteq^*_{\fin\otimes\fin} A_k$ for every $k\in\omega$.
\end{lemma}

In other words, $(\I^{++},
\subseteq^*)$ is $\sigma$-closed.
In a certain sense $\I^{+}$ is even a selective co-ideal, a fact which will be more or less implicit in the proof of Lemma~\ref{exploringTx} below.

\begin{proof}
As in the previous section, we construct two sequences $( B_n)_{n\in \alpha}$ and  $(C_n)
_{n\in \alpha}$ of length $\alpha \leq \omega$ such that for each $n < \alpha$, 
\begin{itemize}
\item $B_n \in (\fin\otimes\fin)^{++}$;
\item $B_n \subseteq A_n$ and $(\forall k\in\omega)\; B_n \subseteq^*_{\fin\otimes\fin} A_k$;
\item $C_n \in \A\setminus\{C_i\mid i<n\}$;
\item $B_n \cap C_n \in (\fin\otimes\fin)^+$ and $B_n \cap C_i \in \fin\otimes\fin$ for $i<n$.
\end{itemize}
Suppose we have found $B_i$ and $C_i$ as above for $i<n$. 
Define a sequence $m^n_0, m^n_1, \hdots$ from $\omega$ by recursion on $k$ as follows:
\[
m^n_k = \min\Big( \dom\big(A_{n+k} \setminus \big(\{ m_i  \mid i < k\} \cup
          \bigcup_{i < n} C_i\big)\big)\Big)
\]
and let $B = \bigcup_{k\in\omega} A_{n+k}(m^n_k)$.

In the case of $\fin(\phi)$, instead chose finite sets $M^n_0, M^n_1, \hdots$ such that 
$M^n_k\subseteq \dom\big(A_{n+k}\setminus\bigcup_{i < n}(C_i \cup M^n_i)\big)$
and $\phi(M^n_k)>0$ for each $k\in\omega$.
Then let 
\[
B=\bigcup_{k\in\omega}\bigcup_{m\in M^n_k} A_{n+k}(m).
\]
The remainder of the proof is essentially identical to the 1-dimensional case, i.e., Lemma~\ref{diagonalization.1D},
simply replacing $\fin$ by $\fin\otimes\fin$ everywhere.
We leave this to the reader.

\end{proof}

\medskip

\subsection{The two-dimensional Branch Lemma}\label{s.2D.branch}

The crucial definition is again that of an invariant tree, analogous to Definition~\ref{d.1D.T^x}.
\begin{definition}\label{d.2D.tree}
For $x \subseteq \omega\times\omega$, let
\[
T^x = \{t\in T\mid (\exists y \in \pi[T_{[t]}])\; y \cap x \notin \fin\otimes\fin\}
\]
\end{definition}
As in Section~\ref{s.1D.branch}, it is easy to see that whenever $x \Delta x' \in \fin\otimes\fin$, $T^x = T^{x'}$.
Moreover Facts~\ref{factsaboutTx}(\ref{factsaboutTx.start})--(\ref{factsaboutTx.end}) hold here as well.

\medskip

We are now ready to state the main lemma of this section.

\begin{blemma}\label{l.2D.branch}
$
\forces_{\M^{\I}_2} |\pi[T^{x_{\dot G}}]| \leq 1.
$
\end{blemma}

We postpone the proof of the Branch Lemma and first give the proof of the Main Proposition \ref{mprop.2D}, assuming the lemma. The proof is not quite as straightforward as in the previous section, but the idea remains the same.

\begin{proof}[Proof of the Main Proposition~\ref{mprop.2D}]
Suppose towards a contradiction there is $p_0\in \M^{\I}_2$ forcing that 
there is $A \in \pi[T]^{V[G]}$ with $A \cap x_G \notin \fin\otimes\fin$.
By the Branch Lemma~\ref{l.2D.branch}, $p_0$ forces that $\pi[T^{x_G}]$ has precisely one element;
let $\dot A$ be a name for it.
\begin{claim}\label{c.inV.dim2}
There is $q\in \M^{\I}_2$ and $A'\in V$ such that $q\forces\dot A=\check A'$.
\end{claim}
\begin{proof}[Proof of Claim.]
It suffices to show that if $p \leq p_0$ and
$p$ decides $(n,m)\in \dot A$ then in fact $(a(p_0),A(p))$ decides $(n,m)\in\dot A$:
For then we may pick $A_0 \supseteq A_1 \supseteq \hdots$ such that for each pair $(n,m) \in \omega\times\omega$, some $(a(p_0),A_k)$ decides $(n,m)\in\dot A$; by Lemma~\ref{diagonalizing.A} we can find $A_\infty$ diagonalizing $(A_k)_{k\in\omega}$.
Any condition below $q=(a(p_0),A_\infty)$ is compatible with each $(a(p_0),A_k)$, and 
so $q$ decides all of $\dot A$.

\medskip

So suppose $p\leq q$ decides $(n,m)\in\dot A$; we must show $(a(p_0),A(p))$ decides $(n,m)\in\dot A$. 
Let us suppose that $p\forces (n,m)\in \dot A$; the proof is similar in case $p\forces (n,m)\notin \dot A$ and we leave this case to the reader. 

Fix any $\M^{\I}_2$-generic  $G$ such that $(a(p_0),A(p))\in G$. 
By Lemma~\ref{l.2D.properties}(\ref{decomposition.2D}) we can decompose $G$ as $G_0 \times G_1$ where $G_1$ is $\M^{\I}_2$-generic and $G_0$ is $\M^k$-generic for $k$ large enough so that $\dom(a(p))\subseteq k$.
Note that as $x_{G} \Delta x_{G_1} \in \fin\otimes\fin$, $T^{x_{G}} = T^{x_{G_1}} \in V[G_1]$. 
Since $V[G]\vDash \pi[T^{x_{G}}]= \{\dot A^{G}\}$, by a simple absoluteness argument the same must hold in $V[G_1]$, i.e., $\dot A^{G} \in \dot V[G_1]$ and $V[G_1]\vDash \pi[T^{x_{G}}]=\{\dot A^{G}\}$.

Since $\dom(a(p)) \subseteq k$ we can find $G'$ which is $\M^{\I}_2$-generic over $V$ such that $G'=G'_0 \times G_1$ and 
$p \in G'$. 
Clearly $(n,m)\in\dot A^{G'}$ (since $p \forces (n,m) \in \dot A$).
Arguing as before using absoluteness, this time between $V[G']$ and $V[G_1]$, 
$\dot A^{G'}$ must equal the unique element of $\pi[T^{x_{G_1}}]$, i.e., $\dot A^{G'}=\dot A^{G}$ and so $(n,m)\in\dot A^{G}$. 
Since $G$ was arbitrary, $(a(p_0), A(p)) \forces (n,m)\in \dot A^{G}$. 
\renewcommand{\qedsymbol}{{\tiny Claim \ref{c.inV.dim2}.} $\Box$}
\end{proof}
Now $A'\in \pi[T] \cap V$ and thus  $A'\in \mathcal I$,  but also
$q \forces x_{\dot G} \cap \check A' \notin\fin\otimes\fin$, contradicting Lemma~\ref{l.2D.properties}(\ref{l.2D.properties.I}).
\renewcommand{\qedsymbol}{{\tiny Main Proposition \ref{mprop.2D}.} $\Box$}
\end{proof}

\medskip

We now gradually work towards the proof of the Branch Lemma, for which it is necessary to introduce some notation.
Firstly, write
\[
U = [\omega\times\omega]^{<\omega}\times T.
\]
Given a pair $\vec u \in U$, we write it as $(a(\vec u), t(\vec u))$ if we
want to refer to the components of $\vec u$.

We define a partial order $\leq_{U}$ on $U$ 
as follows: 
\begin{align*}
\vec u_1 \leq_U \vec u_0 \Leftrightarrow a(\vec u_1) \sqsupseteq_{2} a(\vec u_0)
\land t(\vec u_1) \sqsupseteq t(\vec u_0).
\end{align*}

\medskip

Now secondly assume $G$ is $\M^{\I}_2$-generic over $V$;
working in $V[G]$ for the moment and for a fixed $x \in \powerset(\omega\times\omega)$, define the set
$U^{x}\subseteq U$ consisting of those pairs
$(a,t)\in U$ such that there is $w \in [T_{[t]}]$ with
\begin{enumerate}
\item $\pi(w)\cap x \notin \fin\otimes\fin$, 
\item $\dom(a) \subseteq \dom_{\infty}(\pi(w)\cap x)$ and
\item for each $k \in \dom(a)$, $a(k) \subseteq \pi(w)(k)\cap x(k)$.
\end{enumerate}

Intuitively, $U^{x}$ searches for a branch through $T$ whose projection has large
intersection with $x$ and a subset of this intersection in $(\fin\otimes\fin)^{++}$ to
witness its largeness.

In analogy to the tree $T^x$, when $\vec u_0\in U$ write $U^{x}_{[\vec u_0]}$ for 
$\{\vec u\in U \mid \vec u \leq_U \vec u_0\}$.

\medskip

The following three lemmas gather some observations concerning $U^{x_G}$ which will be important in the proof of the Branch Lemma.

\begin{lemma} \label{exploringTx} Suppose $(a,A) \forces \vec u \in U^{x_{\dot G}}$.
\begin{enumerate}

\item \label{a(u)-subset} It holds that $a \supseteq a(\vec u)$.

\item \label{infiniteDomain} The set $A' \subseteq \omega\times\omega$ defined by
\begin{align*}
A'=\{(k,l)\mid (\exists p' \leq (a,A))(\exists \vec u'\leq_U \vec u)\; (k, l) \in a(\vec u') \land
p'\forces \vec u' \in U^{x_{\dot G}}\}
\end{align*}
is not in $\I$.

\item \label{infiniteExtensionInK}For any $k\in \dom(a(\vec u))$, the set $A_k
\subseteq\omega$ defined by
\begin{align*}
\{l \mid (\exists p' \leq (a,A))(\exists \vec u'\leq_U \vec u)\; l\in a(\vec u')(k) \land
p'\forces \vec u' \in U^{x_{\dot G}}\}
\end{align*}
is not in $\fin$ (resp., in $\fin(\phi_k)$).
\end{enumerate}
\end{lemma}
\begin{proof}
(\ref{a(u)-subset}) Immediate from the definition of $U^{x_{\dot G}}$.

\medskip
(\ref{infiniteDomain})
Assume to the contrary that $A' \in \mathcal I$. Then $A \setminus A' \in \I^{+}$, so
take $B \subseteq A \setminus A'$ such that $B \in \I^{++}$ and set $p = (a,B) \in \M^{\I}_2$.
Since $p \forces \vec u \in U^{x_{\dot G}}$ we can find a name $\dot w$ such that
\[
p \forces \dot w \in \pi[T_{[t(\vec u)]}] \land \dot w\cap x_{\dot G} \notin \fin\otimes\fin.
\]
(In fact, all we need here is that $p\forces T^{x_{\dot G}}\neq\emptyset$).
Thus we can extend $p$ to $p'$ to force a pair $(k,l)$ into $\dot w\cap x_{\dot G}\setminus a(p)$.
But it has to be the case that $(k, l) \in a(p')$, whence $(k, l) \in A'$ by definition of $A'$,
contradicting that also $(k,l) \in B$  which is disjoint from $A'$.

\medskip
(\ref{infiniteExtensionInK})
Assume to the contrary that $k\in \dom(a(\vec u))$ and $A_k \in \mathcal \fin$. Take $B \subseteq A\setminus (\{k\} \times A_k)$ such that $B \in \I^{++}$,
and set $p = (a,B) \in \M^{\I}_2$.
Since $p \forces \vec u \in U^{x_{\dot G}}$ we can find a name $\dot w$ such that
\[
p \forces \dot w \in \pi[T_{[t(\vec u)]}] \land  \dot w\cap x_{\dot G} \in (\fin\otimes\fin)^{+}.
\]
and 
\[
p \forces  \dom(a(\vec u))\subseteq \dom_\infty(\dot w\cap x_{\dot G}).
\]
As $k \in \dom_\infty(\dot w \cap x_{\dot G})$, we can extend $p$ to $p'$ to force a pair $( k,l)$ into $\dot w\cap x_{\dot G}\setminus a(p)$.
But as in the proof of the previous item, it has to be the case that $( k, l ) \in a(p')$, whence $l  \in A_k$ by definition of $A_k$,
contradicting that also $ l\in B(k)$ which is disjoint from $A_k$.
\end{proof}

\medskip

In order to prove the two-dimensional Branch Lemma, we also need to introduce the partially ordered
set $\Gamma$ defined as follows: 
\[
\Gamma = \{(p, \vec u^0, \vec u^1) \in \M^{\I}_2 \times U \times U \mid (\forall
i \in \{0,1\}) \; p \forces \vec u^i \in U^{x_{\dot G}}\}.  \]
This set carries a weak and a strict order, defined as follows:
\[
(p_1, \vec u^0_1, \vec u^1_1) \leq_\Gamma (p_0, \vec u^0_0, \vec u^1_0)
\]
if and only if $p_1 \leq p_0$, and for each $i\in\{0,1\}$, $a(u^i_1) \sqsupseteq_2 a(u^i_0)$ and $t(u^i_1) \sqsupseteq t(u^i_0)$
(that is, $u^i_1 \leq_U u^i_0$);
and 
\[
(p_1, \vec u^0_1, \vec u^1_1) \prec_\Gamma (p_0, \vec u^0_0, \vec u^1_0)
\]
if and only in addition, $a(\vec u_0^0)\cap a(\vec u_0^1) \sqsubset_2 a(\vec
u_1^0) \cap a(\vec u_1^1)$.

Note that $\Gamma$ is well-founded with respect to the second, strict ordering
$\prec_{\Gamma}$; indeed, suppose towards a contradiction that there is an infinite
$\prec_\Gamma$-descending sequence \[
\hdots \prec_\Gamma
(p_3, \vec u^0_3, \vec u^1_3)\prec_\Gamma (p_2, \vec u^0_2, \vec u^1_2) \prec_{\Gamma}
(p_1, \vec u^0_1, \vec u^1_1) \]
from $\Gamma$.
Define
\begin{align*}
y^i &= \bigcup_{n>1} t(\vec u^i_n)\\
\intertext{for $i\in\{0,1\}$ and}
A &= \bigcup_{n>1} a(\vec u^0_n)\cap a(\vec u^1_n).
\end{align*}
Since the sequence is $\prec_\Gamma$-decreasing and from $\Gamma$, $A \in
(\fin\otimes\fin)^{++}$ and $A \subseteq \pi(y^0) \cap \pi(y^1)$, contradicting
that $\pi[T]$ is $\fin\otimes\fin$-almost disjoint.

\medskip

The following lemma says that we can approximate $U^{x_G}$ reasonably well in the ground model.
A very similar proof shows that $\M^{\I}_2$ is proper.
\begin{lemma}\label{l.force.finite.bit}
For each $\vec u_0 \in U$ the set $D(\vec u_0)$ is dense and open in
$\M^{\I}_2$, where we define $D(\vec u_0)$ to be the set of $p\in \M^{\I}_2$ such that for
all $p' \leq p$ and any $\vec u \in U$,
\[
\left[\, p' \forces\, \vec u \in U^{x_{\dot G}}_{[\vec u_0]}\, \right] \Rightarrow
(a(p'),A(p)/a(p')) \forces (\exists t \in T) \; (a(\vec u)),t) \in
U^{x_{\dot G}}_{[\vec u_0]}.
\]
\end{lemma}

The proof follows the same strategy as Lemma~\ref{diagonalizing.A} (the diagonalization lemma) to build a set in $\mathcal I^{++}$.
While we build this set, we carefully anticipate each of its finite subsets $a$ to see if there is some $t\in T$ and some forcing condition $q \in \M^{\I}_2$
which forces $(a, t)$ to be in
$U^{x_{\dot G}}$. 
If so, we make sure that our final set is contained in $a\cup A(q)$. We succeed as there are only countably many finite $a \subseteq \omega\times\omega$ to consider.
Note though that due to the nature of the proof of Lemma~\ref{diagonalizing.A}, we have to consider each finite $a$ again and again, and the construction potentially takes $\omega\times\omega$ stages.
\begin{proof}
Fix $q_0 \in \M_2^{\I}$ and $\vec u_0 \in U$ such that $q \forces \vec u_0 \in U^{x_{\dot G}}$.
We construct $q \leq q_0$ such that $q \in D(\vec u_0)$.

As in the proof of Lemma~\ref{diagonalizing.A}, we construct sequences $B_0, B_1, \hdots,$ and $C_0, C_1, \hdots$ both of which are possibly finite, such that whenever defined
\begin{itemize}
\item $B_n  \in (\fin\otimes\fin)^{++}$;
\item $C_n \in \A\setminus\{C_i \mid i <n\}$;
\item $B_n\cap C_n \in (\fin\otimes\fin)^{+}$ while for $i<n$, $B_n\cap C_i \in \fin\otimes\fin$.
\end{itemize}
Suppose $B_i$ and $C_i$ have been defined for $i<n$ (this includes the case $n=0$).
In $\omega$-many steps we define a descending sequence of conditions
$(b^k_n, B^k_n)_k$  from $\M^{\I}_2$ and at the end let
\begin{equation}\label{e.def.B_n}
B_n =  \bigcup_{k\in\omega} b^k_n.
\end{equation}
If $n=0$, let $b^0_0 = a(q_0)$ and $B^0_0 = A(q_0)$.
Otherwise, let 
\[
b^0_n = \bigcup_{i,j<n} b^i_j 
\]
and
\[
B^0_n = \big(B^{n-1}_{n-1} \res \dom(b^0_n)\big) \cup B^{n-1}_{n-1} \setminus \bigcup \{ C_i \setdef i<n\}
\]
noting that $B^0_n \in (\fin\otimes\fin)^+$ since $B^{n-1}_{n-1} \in \mathcal I^{++}$ by induction hypothesis.
So $(b^0_n, B^0_n) \in \M^{\I}_2$ and $B^0_n \cap C_i \in \fin\otimes\fin$ for $i<n$.
 
 \medskip
 
Supposing we have already defined $(b^k_n, B^k_n)  \in \M^{\I}_2$ we thin out $B^k_n$ to $B^*\in \mathcal I^{++}$ in finitely many steps such that whenever $a \subseteq b^k_n$ and 
 \begin{equation}\label{e.u}
 ( \exists p' \leq (a(q_0), B^k_n)) (\exists \vec u\in U) \; a(\vec u) ) = a \wedge a(p')  \subseteq b^k_n \wedge p' \forces \vec u \in U^{x_{\dot G}}_{[\vec u_0]}
 \end{equation}
then for some $t' \in T$
  \begin{equation}\label{e.B^*}
(a(p'), B^*) \forces (a(\vec u), t') \in U^{x_{\dot G}}_{[\vec u_0]}.
 \end{equation}
Extend $b^k_n$ to the some finite (we mean finite also in the general case!) set $b^{k+1}_n\subseteq\omega\times\omega$
satisfying
\begin{equation}\label{e.B_n.grows}
b^k_n \sqsubset_2 b^{k+1}_n \subseteq b^k_n \cup B^*
\end{equation}
and let 
\[
B^{k+1}_n = B^*/ b^{k+1}_n.
\]
Assuming we have defined $b^k_n$ for each $k\in \omega$ and letting $B_n$ be defined by \eqref{e.def.B_n},  note that \eqref{e.B_n.grows} ensures that $B_n \in (\fin\otimes\fin)^{++}$.
Should it be the case that $B_n \in \I^{+}$ the construction terminates and we let
\[
q=(a(q_0),B_n).
\]
Otherwise, we may chose $C_n \in \A\setminus\{C_i \mid i< n\}$ such that $C_n \cap  B_n \in (\fin\otimes\fin)^+$ as in Lemma~\ref{diagonalizing.A}
and continue the construction.

\medskip

If the construction does not terminate at any stage $n<\omega$, let
\[
B_\infty= \bigcup_{n\in\omega} b^n_n.
\]
Note that $B_\infty = \bigcup_{k\in\omega} B_k$ and thus since $B_\infty \cap C_k\in(\fin\times\fin)^+$ for each $k\in\omega$,
it must be the case that $B_\infty \in \mathcal I^{++}$ (as in the proof of Lemma~\ref{diagonalizing.A}). 
So  we obtain a condition  in $\M^{\I}_2$ by letting
\[
q=(a(q_0),B_\infty).
\]

\medskip

To see that $q \in D(\vec{u}_0)$, let $p' \leq q$, $\vec{u} \in U$ such that  
$p' \forces \vec u \in U^{x_{\dot G}}_{[\vec u_0]}$ be given.
Let us first assume that the construction did not stop at any stage $n<\omega$ and that $B_\infty$ is defined.
We can find $n>0$ so that $a(p') \subseteq b^{n-1}_{n-1}$. 
Thus, at stage $k=n$ in the construction of $B_{n}$, \eqref{e.u} was satisfied for $a=\vec{u}$, and so \eqref{e.B^*} is also satisfied.
By construction $B_\infty\setminus b_{n-1}^{n-1} \subseteq B^{n}_{n}$. 
Thus any condition below $(a(p'), B_\infty) = (a(p'), A(q))$ is compatible with $(a(p'),B^{n}_{n})$, and so we may replace $B^*$ by $A(q)$ in \eqref{e.B^*},
obtaining
 \begin{equation*}
(\exists t \in T)\; (a(p'), A(q)) \forces \vec u \in U^{x_{\dot G}}_{[\vec u_0]}
 \end{equation*}
and showing that $q \in D(\vec{u})$.
 
 If the construction of $B_0, B_1, \hdots$ terminated with $B_n \in \mathcal I^{++}$, we may find $k$ such that $a(p') \subseteq b^{k-1}_n$ and argue similarly with $B_n$ in place of $B_\infty$.
\end{proof}

\medskip

The proof of the Branch Lemma will crucially depend on the following simple lemma.
It plays the same role as 
Lemma~\ref{l.1D.properties}(\ref{finitepartnegligible})
in that it allows us to change the finite part of a condition while maintaining that something is forced about $U^{x^{\dot G}}$.
\begin{lemma}\label{l.force.intersection}
For any $p\in \M^{\I}_2$, $\vec u \in U$ such that 
$p \forces \vec u \in U^{x_{\dot G}}$, any $a\subseteq a(p)$ and $a'
\subseteq a \cap a(\vec u)$, it holds that \[
(a,A(p)/a) \forces  (a',t(\vec u)) \in U^{x_{\dot G}}.
\]
\end{lemma}
\begin{proof}
Let $G$ be a generic over $V$ with $(a,A(p)/a) \in G$, and let 
\[
I = \dom(a(p)) \setminus \dom(a).
\] 
Suppose $H$ is $\prod_{j\in I} \M$-generic over $V[G]$ such that $(a(p)(j),
A(p)(j))_{j\in I} \in H$. 
Then $G\times H$ is generic over $V$ for 
\[
\M_2^{\I}\times \prod_{j\in I}\M .
\]
We define a bijection
\begin{align*}
\phi\colon \M_2^{\I}\Big(\leq \big(a',A(p)/a'\big)\Big)\times \prod_{j\in I}\M \Big(\leq \big(a(p)(j), A(p)(j)\big)\Big)
   \to \M_2^{\I}(\leq p)
\end{align*}
by
$((b,B),(c_j,C_j)_{j\in I}) \mapsto (a^*, A^*)$ where
\begin{align*}
a^*(k) &=
\begin{cases}
a(p)(k) \cup \big(b(k)\setminus a(k)\big) & \text{ if } k\in \dom(a);\\
c_k & \text{ if } k\in I; \\
b(k) & \text{ otherwise}.
\end{cases}
\intertext{and}
A^*(k) &=
\begin{cases}
C_k & \text{ if } k\in I; \\
B(k) & \text{ otherwise}.
\end{cases}
\end{align*}

\medskip

Note that $p\in \phi(G\times H)$, so $\vec u\in U^{x_{\phi(G\times H)}}$ in $V[G][H]$.
By definition of $U^{x}$ %
this means that in $V[G]$ we can find $w \in [T_{[t(\vec u)]}]$
so that 
\begin{equation}\label{e.existsw}
(\exists u \in (\fin\otimes\fin)^{++})\;a \subseteq u \subseteq \pi(w) \cap x_{\phi(G\times H)}.
\end{equation}
Since $a'' \subseteq x_{G}$  and since $x_G\Delta x_{\phi(G\times H)} \in \fin\otimes\fin$ we may replace $a$ by $a'$ and then $x_{\phi(G\times H)}$ by $x_{G}$ in \eqref{e.existsw}, and thus  
\begin{equation}\label{e.x-v}
(\exists x \in \pi[T_{[t(\vec u)]}])(\exists u \in (\fin\otimes\fin)^{++})\; a'' \subseteq  u \subseteq \pi(x) \cap x_{G}.
\end{equation}
It is easy to find a tree $S \in V[G]$ such that $[S]$ consists of the (codes for) pairs $(x, u)$ witnessing the two existential quantifiers in \eqref{e.x-v}. 
 Since being
well-founded is absolute between models of $\ZFC$, we conclude \eqref{e.x-v} holds in $V[G]$. 
But \eqref{e.x-v} implies (in fact, is equivalent to) $ (a',t(\vec u)) \in U^{x_{G}}$, so since $G$ was arbitrary,  we have shown that $(a',A/a')\forces (a'',t(\vec u)) \in U^{x_{\dot G}}$.
\end{proof}

With this notation and the lemmas at our disposal, we are ready to prove 
$$
\forces_{\M^{\I}_2} \lvert \pi[T^{x_{\dot G}}] \rvert \leq 1,
$$
i.e., 
the Branch
Lemma~\ref{l.2D.branch}.

\begin{proof}[Proof of the Branch Lemma~\ref{l.2D.branch}]
Assume towards a contradiction that the lemma is false,
whence we may find $p \in \M^{\I}_2$ and a pair of $\M^{\I}_2$-names $\dot w^0$ and $\dot
w^1$ so that
\[
p \forces (\forall i \in \{0,1\}) \; \dot w^i \in [T^{x_{\dot G}}] \wedge  {x_{\dot G}\cap\pi(\dot w^i)}
\notin \fin\otimes\fin
\]
and $p \forces \pi(\dot w^0) \neq \pi(\dot w^1)$.
Then clearly we may also find $(p_0, \vec u^0_0, \vec u^1_0) \in \Gamma$ such that
$\pi(t(\vec u^0_0)) \neq \pi(t(\vec u^1_0))$ ($a(\vec u^i_0)$ plays no
role here).

\begin{claim}\label{c.alternative}
One of the following holds:
\begin{enumerate}
\item\label{alt.horizontal} There is $n^* \in \omega$ and $(p_1, \vec u^0_1, \vec u^1_1)
\leq_\Gamma (p_0, \vec u^0_0, \vec u^1_0)$ in $\Gamma$ such that for any $(p_2, \vec
u^0_2, \vec u^1_2)\leq_\Gamma (p_1, \vec u^0_1, \vec u^1_1)$ from $\Gamma$,
$\dom(a(\vec u^0_2)) \cap \dom(a(\vec u^1_2)) \subseteq n^*$; or
\item\label{alt.vertical} There is $(p_1, \vec u^0_1, \vec u^1_1)  \leq_\Gamma (p_0, \vec
u^0_0, \vec u^1_0)$, $l^* \in \omega$ and $k^* \in \dom(a(\vec u^0_1)) \cap \dom(a(\vec u^1_1))$
such that for any $(p_2, \vec u^0_2, \vec u^1_2)\leq_\Gamma (p_1, \vec u^0_1, \vec u^1_1)$
from $\Gamma$,
\[
a(\vec u^0_2)(k^*)\cap  a(\vec u^1_2)(k^*) \subseteq l^*.
\]
\end{enumerate}
\end{claim}

\begin{proof}[Proof of Claim]
Suppose that both Items~\ref{alt.horizontal} and~\ref{alt.vertical} above
fail; we show that there is a $\prec_\Gamma$-descending sequence in $\Gamma$, which contradicts the
wellfoundedness of $(\Gamma, \prec_{\Gamma})$. 

It suffices to
show that any $(p, \vec u^0, \vec u^1) \leq_\Gamma (p_0, \vec u^0_0, \vec u^1_0)$
has a $\prec_\Gamma$-extension.
That Items~\ref{alt.horizontal} and~\ref{alt.vertical} above
fail means precisely that
\begin{enumerate}[label=(\arabic*')]
\item\label{alt.horizontal.fail} For each $n \in \omega$ and $(p_1, \vec u^0_1, \vec u^1_1)
\leq_\Gamma (p_0, \vec u^0_0, \vec u^1_0)$ in $\Gamma$ there is $n^* > n$ and $(p_2, \vec
u^0_2, \vec u^1_2)\leq_\Gamma (p_1, \vec u^0_1, \vec u^1_1)$ from $\Gamma$ such that
$n \in \dom(a(\vec u^0_2)) \cap \dom(a(\vec u^1_2))$; and
\item\label{alt.vertical.fail} For each $(p_1, \vec u^0_1, \vec u^1_1)  \leq_\Gamma (p_0, \vec
u^0_0, \vec u^1_0)$,  $k^* \in \dom(a(\vec u^0_1)) \cap \dom(a(\vec u^1_1))$ and $l \in \omega$ there is $l^* > l$ and
$(p_2, \vec u^0_2, \vec u^1_2)\leq_\Gamma (p_1, \vec u^0_1, \vec u^1_1)$
from $\Gamma$ such that 
\[
l \in a(\vec u^0_2)(k^*)\cap  a(\vec u^1_2)(k^*).
\]
\end{enumerate}
This means that in finitely many steps, we can extend any $(p, \vec u^0, \vec u^1)\leq_\Gamma (p_0, \vec u^0_0, \vec u^1_0)$ to some $(q, \vec v^0, \vec v^1) \leq_\Gamma (p, \vec u^0, \vec u^1)$ so that 
\[
a(\vec u^0) \cap a(\vec u^1) \sqsubset_2 a(\vec v^0) \cap a(\vec v^1)
\] 
by applying \ref{alt.vertical.fail} once for each vertical in $a(\vec u^0) \cap a(\vec u^1)$ and \ref{alt.horizontal.fail} once for the domain. 
Thus 
$(q, \vec v^0, \vec v^1) \prec_\Gamma (p, \vec u^0, \vec u^1)$.
\end{proof}

\medskip

Finally, having established that one of Items~\ref{alt.horizontal} and~\ref{alt.vertical}
above must hold, we use Lemmas~\ref{l.force.intersection} and \ref{l.force.finite.bit} to
finish the proof of Lemma~\ref{l.2D.branch} by case distinction.

\medskip

\textbf{Case 1:} If Item~\ref{alt.vertical} holds, we may fix  $(p_1, \vec u^0_1, \vec
u^1_1)\in\Gamma$, $l^* \in \omega$ and $k^* \in \dom(a(\vec u^0_1)) \cap \dom(a(\vec u^1_1))$ such that
for any $(p_2, \vec u^0_2, \vec u^1_2)\leq_\Gamma (p_1, \vec u^0_1, \vec u^1_1)$ from
$\Gamma$,
\[
a(\vec u^0_2)(k^*)\cap  a(\vec u^1_2)(k^*) \subseteq l^*.
\]
We may also assume that $p_1 \in D(\vec u^0_1) \cap D(\vec u^1_1)$ (see
Lemma~\ref{l.force.finite.bit}).  We now reach a contradiction:
Define $A \subseteq \omega\times\omega$ by letting $A(k) = A(p_1)(k)$ for each $k\neq k^*$, and letting
\begin{equation*}
A(k^*)=\{ l\in\omega \mid (\exists p \leq p_1) (\exists \vec u \leq_U \vec u^0_1) \; l \in
a(\vec u)(k^*) \wedge p \forces \vec u \in U^{x_{\dot G}}\}.
\end{equation*}
Lemma \ref{exploringTx} ensures that $A \in \I^{++}$.
Let
\[
p^* = (a(p_1), A).
\]
Since $p^* \forces \vec u^1_1 \in U^{x_{\dot G}}$,
we can find $p \leq p^*$, $l \in \omega \setminus l^*$ and $\vec u$ such that
\[
l \in a(\vec u)(k^*) \wedge p \forces  \vec u \in U^{x_{\dot G}}_{[\vec u^1_1]}.
\]
It follows that $l\in A(k^*)$ and so by definition of $A(k^*)$ we can find $p' \leq p_1$
and $\vec u'$ such that
\[
l \in a(\vec u')(k^*) \wedge p' \forces  \vec u' \in U^{x_{\dot G}}_{[\vec u^0_1]}.
\]
Then, as $p,p' \leq p_1$ and $p_1 \in D(\vec u^0_1) \cap D(\vec u^1_1)$,
we can find $\vec u^0$ and $\vec u^1$ such that
\[
l \in a(\vec u^0)(k^*) \wedge (a(p),A(p_1)/a(p)) \forces  \vec u^0 \in U^{x_{\dot G}}_{[\vec u^0_1]}
\]
and
\[
l \in a(\vec u^1)(k^*) \wedge (a(p'),A(p_1)/a(p')) \forces  \vec u^1 \in U^{x_{\dot
G}}_{[\vec u^1_1]}.
\]
Note that $\{(k^*,l)\}\cup a(\vec u_1^i)\subseteq
a(p_1) \initseg a(p) \cap a(p')$ for each $i\in\{0,1\}$. 
Let $a=a(p) \cap a(p')$, let $p_2=(a, A(p_1)/a)$ and let $a^i = a(\vec u_1^i) \cup
\{(k^*,l)\}$ for each $i\in\{0,1\}$. 
By Lemma~\ref{l.force.intersection} we
conclude
\begin{align*}
p_2 \forces (a^i, t(\vec u^i)) \in U^{x_{\dot G}}_{[\vec u_1^i]},
\end{align*}
which contradicts the choice of $(p_1, \vec u_1^0, \vec u_1^1)$ and $l^*$.

\medskip

\textbf{Case 2:} Otherwise, Item~\ref{alt.horizontal} holds and we may fix  $n^* \in \omega$
and $(p_1, \vec u^0_1, \vec u^1_1)  \in\Gamma$ such that for any $(p_2, \vec u^0_2, \vec
u^1_2)\leq_\Gamma (p_1, \vec u^0_1, \vec u^1_1)$ from $\Gamma$,
\begin{align*}
\dom(a(\vec u^0_2)) \cap \dom(a(\vec u^1_2)) \subseteq n^*.
\end{align*}

We now argue entirely analogously to the previous case, but in the domain instead of in
one of the verticals. To this end, set
\[
A'=\{( k,l)\mid (\exists p \leq p_1)(\exists \vec u\leq_U \vec u^0_1)\; ( k, l) \in a(\vec u') \land
p \forces \vec u \in U^{x_{\dot G}}\}.
\]
Note that $A' \subseteq A(p_1)$ and
$A \in \I^{+}$ by Lemma \ref{exploringTx} (\ref{infiniteDomain}). Let $A \subseteq A'$ be
the largest subset satisfying $A \in \I^{++}$.
Letting $p^* = (a(p_1), A)$ we reach a contradiction almost exactly as in the previous case; details are left to the reader.
\renewcommand{\qedsymbol}{{\tiny Lemma \ref{l.2D.branch}.} $\Box$}
\end{proof}


\medskip

\section{Iterated Fubini products}\label{s.alphaD}

In this section we will look at iterated Fubini products of $\fin(\phi)$-ideals.
In order to study these, we will first recursively define sets $M^{\alpha}$:

\begin{definition} \label{domain.fin.alpha}
Set $M^{1} = \omega$.
For a successor ordinal, set $M^{\alpha + 1} = \omega \times M^{\alpha}$. For
$\alpha$ limit ordinal, fix once and for all a sequence $(\alpha_n)_{n\in\omega}
\subseteq \alpha$ which is cofinal in $\alpha$, and set
$M^{\alpha} = \bigcup_{n\in\omega} \{n\} \times M^{\alpha_n}$.
\end{definition}


We will fix some notation concerning the sets $M^{\alpha}$:
\begin{notation}
Let $X \subseteq M^{\alpha}$.
We set 
\[
X(n) = \{ x \in \bigcup_{\beta < \alpha} M_{\beta} \mid (n,x) \in X\}.
\]
For
$\alpha > 1$, we
let as usual 
$
\dom(X) = \{n\in\omega \mid X(n) \neq \emptyset\}.
$
If $(n_0, \dots, n_k) \in \omega^{k+1}$ satisfies that $n_0 \in \dom(X)$
and for every $1 \leq i
< k$ we have $n_i \in \dom(X(n_0)\cdots(n_{i-1}))$ and $n_{k} \in X(n_0)\cdots(n_k)
\subseteq \omega$, we say that $(n_0, \dots, n_{k})$ is a \emph{terminal sequence}.
Any proper initial segment of a terminal sequence in X is called 
a \emph{domain sequence} in $X$. 
 Note that we allow a
domain sequence to be empty, and set $X(\emptyset) = X$. We will often refer to domain
sequences and terminal sequences as vectors, $\vec n = (n_0, \dots, n_k)$, and we will write
$X(\vec n)$ for $X(n_0)\cdots(n_k)$, $\vec n(i)$ for $n_i$, $\vec n\res l$ for $(n_0, \hdots, n_{l-1})$ when $1\leq l\leq k+1$ and $\lh(\vec n)$ for $k+1$, of course setting $\vec n\res 0 = \emptyset$ and $\lh(\emptyset)=0$.
We also set 
\[
\dom_{\alpha}(X) = \{\vec n\res l \mid l < \lh(\vec n)\land\vec n\in  X\}
\] 
i.e., the set of domain sequences in $X$.
We denote by
$\delta_{\alpha}(\vec n)$ the ordinal $\delta \leq \alpha$ such that $X(\vec n) \subseteq
M^{\delta}$. If the origin of the domain sequence $\vec n$ is unambiguous, we will often just
write $\delta(\vec n)$.
\end{notation}

We now define a hierarchy of ideals which complexity-wise lies cofinal in the Borel hierarchy:
\begin{definition}
We define an ideal $\fin^\alpha$ on $M^\alpha$ for $\alpha \in \omega_1\setminus\{0\}$ by recursion as follows:
\begin{itemize}
\item $\fin^1=\fin$.
\item For a successor ordinal $\alpha + 1>1$, set
\begin{align*}
A \in \fin^{\alpha + 1} \Leftrightarrow \{n\in\omega
\mid A(n) \notin \fin^{\alpha}\}\in \fin.
\end{align*}
\item For a limit ordinal $\alpha$ with cofinal sequence $(\alpha_n)_{n\in\omega})$, set
\begin{align*}
A \in \fin^{\alpha} \Leftrightarrow \{n\in\omega \mid
A(n) \notin \fin^{\alpha_n} \}
\in\fin.
\end{align*}
\end{itemize}
\end{definition}

Generalizing the previous definition, we also define iterated Fubini products of a sequence of $F_\sigma$ ideals on $\omega$ (given as the finite part of a submeasure):
\begin{definition}  
We define an ideal
$\fin^{\alpha}(\vec \phi)$ on $M^\alpha$, where $\vec \phi = (\phi_{\beta})_{0<\beta \leq\alpha}$ is a sequence of lsc submeasures on $\omega$ and $\alpha\in\omega_1\setminus\{0\}$. 
The definition is again by recursion on $\alpha$:
\begin{itemize}
\item For $\alpha =1$ and $A \subseteq M^1$ set
\begin{align*}
A \in \fin(\vec \phi) &\Leftrightarrow A \in \fin(\phi_{1}).
\intertext{\item For a successor ordinal $\alpha>1$ and $A \subseteq M^\alpha$ set}
A \in \fin(\vec \phi) &\Leftrightarrow \{n\in\omega
\mid A(n) \notin \fin(\vec \phi\res \alpha)\} \in \fin(\phi_{\alpha}).
\intertext{\item For a limit ordinal $\alpha$  with cofinal sequence $(\alpha_n)_{n\in\omega})$ and $A \subseteq M^\alpha$ set}
A \in \fin(\vec \phi) &\Leftrightarrow \{n\in\omega \mid
A(n) \notin \fin(\phi\res\alpha_n+1) \}
\in\fin(\phi_\alpha).
\end{align*}
\end{itemize}
Clearly $\fin^\alpha=\fin(\vec \phi)$ where for each $\beta$, $\phi_\beta$ is just the counting measure. 
\end{definition}

(One could think of defining yet more general ideals of the form $\fin(\vec \phi)$ on $M^\alpha$ where $\vec \phi=(\phi_s)_{s\in  D(\alpha)}$ is 
an assignment of submeasures to the set $D(\alpha)$ of domain sequences in $M^\alpha$, letting $D(1)=\{\emptyset\}$.
Write 
\[
\vec \phi(n)=(\phi_{n\conc t})_{t \in D(\alpha_n)}, 
\]
where if $\alpha$ is a limit ordinal, $(\alpha_n)_{n\in\omega}$ is its cofinal sequence and if $\alpha$ is a successor, we let $\alpha_n=\alpha-1$. 
We can define $\fin(\vec \phi)$ by recursion on $\alpha$ as follows:
For $\alpha = 1$, let 
$X \in \fin(\vec \phi) \iff \phi_\emptyset(X)<\infty$;
for $\alpha>1$, let 
$X \in \fin(\vec \phi) \iff \phi_\emptyset(\{n\in\omega\mid X(n)\in \fin(\vec \phi(n))\})<\infty$.
We conjecture all our proofs go through.)

\medskip
Since we can view any element in $M^{\alpha}$ as a finite sequence in $\omega$, the set
$M^{\alpha}$ can be identified with a subset of $\omega^{<\omega}$---to be precise, with the set of terminal sequences in $M^{\alpha}$. 
Note that a set $a
\subseteq M^{\alpha}$ is finite if and only if there are finite sets $K_0,\dots\,K_{n-1}$
with $K_i \subseteq \omega$ such that $a \subseteq K_0 \times K_1 \times
\cdots \times K_{n-1}$ under this identification. Furthermore, there is a natural ordering on
$M^{\alpha}$, namely the lexicographical ordering, $\leq_{lex}$, inherited from $\omega^{<\omega}$.
We will also consider several other orderings on $M^{\alpha}$:

\begin{definition}
We recursively define $\sqsubseteq_{\alpha}$ on $M^{\alpha}$ as follows:
\begin{itemize}
\item Set $X \initseg_1 Y$ if and only if $X \initseg Y$, i.e.
if $X$ is an initial segment of $Y$.
\item Set $X \initseg_{\alpha + 1} Y$
if and only if $\dom(X) \initseg \dom(Y)$ and
for every $i\in\dom(Y)$ we have $X(i) \initseg_{\alpha} Y(i)$;
\item For $\alpha$ a limit ordinal with cofinal sequence $(\alpha_n)_{n\in\omega}$, we set
$X \sqsubseteq_{\alpha} Y$ if and only if $\dom(X) \sqsubseteq \dom(Y)$ and for every
$i\in\dom(Y)$ we have $X(i) \sqsubseteq_{\alpha_i} Y(i)$.
\end{itemize}

In order to determine if a set properly extends another set, we need a strict ordering
$\sqsubset_{\alpha}$ on $M^{\alpha}$ to be a version
of $\initseg_{\alpha}$ which is strict at every level.
For the case $\mathcal J = \fin^\alpha$ we make the following definition:
\begin{itemize}
\item Set $X \sqsubset_1 Y$ if and only if $X \sqsubsetneq Y$, i.e.
if $X$ is a proper initial segment of $Y$.
\item Set $X \sqsubset_{\alpha + 1} Y$
if and only if $\dom(X) \sqsubsetneq \dom(Y)$ and
for every $i\in\dom(X)$ we have $X(i) \sqsubset_{\alpha} Y(i)$;
\item For $\alpha$ a limit ordinal with cofinal sequence $(\alpha_n)_{n\in\omega}$, we set
$X \sqsubset_{\alpha} Y$ if and only if $\dom(X) \sqsubsetneq \dom(Y)$ and for every
$i\in\dom(Y)$ we have $X(i) \sqsubset_{\alpha_i} Y(i)$.
\end{itemize}

In the general case of an ideal $\mathcal J=\fin(\vec \phi)$ on $M^\alpha$,
we define $\sqsubset_{\alpha}$ on $M^{\alpha}$ by recursion on $\alpha$ as follows:
\begin{itemize}
\item Set $X \sqsubset_1 Y$ if and only if $X\sqsubseteq Y$ and $\phi_1(X) < \phi_1(Y)$.
\item Set $X \sqsubset_{\alpha + 1} Y$
if and only if $\dom(X) \sqsubseteq \dom(Y)$, 
\[
\phi_{\alpha+1}(\dom(X)) < \phi_{\alpha+1}(\dom(Y)),
\] 
and
for every $i\in\dom(X)$ we have $X(i) \sqsubset_{\alpha} Y(i)$;
\item For $\alpha$ a limit ordinal with cofinal sequence $(\alpha_n)_{n\in\omega}$, we set
$X \sqsubset_{\alpha} Y$ if and only if $\dom(X) \sqsubseteq \dom(Y)$,
$\phi_\alpha(\dom(X)) < \phi_\alpha(\dom(Y))$,
 and for every
$i\in\dom(Y)$ we have $X(i) \sqsubset_{\alpha_i} Y(i)$.
\end{itemize}
As was the case for the previous section, the material of the present section generalizes almost mechanically from $\fin$ to $\fin^\alpha$.
Often this is made possible by of the above definition of $\sqsubset_\alpha$.

\medskip

When defining the $\alpha$-dimensional Mathias forcing notion, we will need an ordering
$<_{\alpha}$ on $M^{\alpha}$ defined as follows:
\begin{itemize}
\item Set $X <_1 Y$ if and only if $\max(X) < \min(Y)$.
\item Set $X <_{\alpha + 1} Y$
if and only if $\dom(X) \sqsubsetneq \dom(Y)$, and for every $i\in\dom(X)$ we have $X(i)
<_{\alpha} Y(i)$.
\item For $\alpha$ a limit ordinal with cofinal sequence $(\alpha_n)_{n\in\omega}$, we set
$X <_{\alpha} Y$ if and only if $\dom(X) \sqsubsetneq \dom(Y)$ and for every
$i\in\dom(Y)$ we have $X(i) <_{\alpha_i} Y(i)$.
\end{itemize}
\end{definition}

We let as usual $(\fin^{\alpha})^+$ denote the co-ideal.

\medskip

The $\alpha$-dimensional forcing notion is now defined as follows:
\begin{definition} Let $(\fin^{\alpha})^{++}$ denote the set of $A \subseteq M^\alpha$
such that for every $\vec n\in \dom_\alpha(A)$ we have $A(\vec n) \notin
\fin^{\delta_{\alpha}(\vec n)}$.
Conditions of $\M_{\alpha}$ are pairs $(a,A)$ where
\begin{enumerate}[label=(\alph*)]
\item $a \subseteq M^{\alpha}$ is finite;
\item\label{r.forcing.alpha} $A \in (\fin^{\alpha})^{++}$;
\item $a <_{\alpha} A$.
\end{enumerate}
We let $(a',A') \leq (a,A)$ if and only if $A'\subseteq A$ and $a \initseg_{\alpha} a'
\subseteq a \cup A$.

\medskip
For the general case, define $\fin(\vec \phi)^{++}$  to be the set of $A \subseteq M^\alpha$ such that
such that for every $\vec n\in \dom_\alpha(A)$ we have $A(\vec n) \notin \fin(\vec \phi\res \delta(\vec n)+1)$.
and replace \ref{r.forcing.alpha} by $A \in \fin(\phi)^{++}$ in the definition of $\M_\alpha$.
\end{definition}

Note that for any $\vec n \in \dom_{\alpha}(a)$, the pair $(a(\vec n), A(\vec n))$ is a
forcing condition in $\M_{\delta_{\alpha}(\vec n)}$. The pair $(\dom(a), \dom(A))$ is a
classical (1-dimensional) Mathias forcing condition. As before, we need a relativized
forcing notion:

\begin{definition}
If $\I^+$ is the co-ideal of an ideal $\I\supseteq\fin^{\alpha}$, then we write
$\I^{++}$ for $\I^{+} \cap (\fin^{\alpha})^{++}$ and we let
$$
\M_{\alpha}^{\I}=\{(a,A)\in \M_{\alpha} \mid A\in\I^{++}\}.
$$
\end{definition}

Note that if $\I=\fin^{\alpha}$ then $\M_{\alpha}^{\I}=\M_{\alpha}$. Note furthermore that
if $A \in \I^+$, then we can always find $B \subseteq A$ such that $B \in \I^{++}$.

\begin{notation}~
\begin{enumerate}
\item
For any $X \in M^{\alpha}$, we define the \textit{generalized infinity domain} by
$\dom_{\alpha}^{\infty}(X) = \{\vec n\in \dom_{\alpha}(X) \mid X(\vec n) \notin
\fin^{\delta_{\alpha}(\vec n)}\}$, and note that $A \in (\fin^{\alpha})^{++}$ if and only
if $\dom_{\alpha}(A) = \dom_{\alpha}^{\infty}(A)$.
\item
Given a filter $G$ on $\M_{\alpha}^{\mathcal I}$, let
$$
x_G=\bigcup\{a \mid(\exists A) (a,A)\in G\}
$$
We will see that for $\M^{\I}_\alpha$-generic $G$, $x_G\in (\fin^{\alpha})^{++}$ holds in $V[G]$.
\item
For a condition $p\in\M^{\I}_{\alpha}$, we write $(a(p), A(p))$ when we want to refer to
its components.
\item
For $(a,A) \in \M^{\I}_{\alpha}$ and $b \subseteq A$ finite, let
\begin{align*}
A/b = \bigcup_{\vec n \in \dom_{\alpha}(A)} A(\vec n) \setminus \big\{x \in M^{\delta_{\alpha}(\vec
n)} \mid \big(\exists \vec m \in b(\vec n)\big) \; x \leq_{lex} \vec m\big\}.
\end{align*}

\item
For $p\in\M^{\I}_{\alpha}$, we let $\M^{\I}_{\alpha}(\leq p) = \{q\in\M^{\I}_{\alpha}
\mid q \leq p\}$.
\end{enumerate}
\end{notation}
\begin{remark}
The definition of $A/b$ was made to guarantee $b <_\alpha A/b$.
Note that $\vec n \in A/b$ if and only if $\vec n \notin b$ and letting $\vec n \res l$ be the longest common initial segment of $\vec n$ with some element of $b$, then there is no $\vec m\in b$ with $\vec n \res l \sqsubseteq \vec m$ and $\vec n(l) < \vec m(l)$.
\end{remark}

\medskip

Following the same strategy as in previous sections, our main pursuit will be a generalization of the Main Proposition \ref{mprop.2D}.

\begin{remark}
Recall that in order to meaningfully talk about $\kappa$-Suslin sets in $\powerset(M^\alpha)$, we identify $M^\alpha$ with $\omega$ (via some fixed arbitrary bijection), sets with their characteristic functions, and in effect, $\powerset(M^\alpha)$ with $2^\omega$ (as described in Section~\ref{s.trees}).
\end{remark}

\begin{assumption}
For the remainder of this article, let $\mathcal J=\fin^\alpha$ where $\alpha \geq 2$ (or more generally, $\mathcal J= \fin(\vec \phi)$.
Suppose $\A \subseteq \powerset(M^\alpha)$ to be a
$\mathcal J$-almost disjoint family which is $\kappa$-Suslin. 
Moreover, fix a tree $T$ on $2\times \kappa$  
such that $\pi[T] = \A$. Finally, let $\I$ be the ideal generated by $\A \cup \fin^\alpha$.
We leave it to the reader to make trivial substitutions to adapt the proofs to the case of $\fin(\vec \phi)$-AD families, but do give details when the proofs differ substantially.
\end{assumption}
Although the proofs in this section work for $\fin(\vec \phi)$ as above we will of notational concern only consider the case where
$\phi_{\beta}$ is the counting measure for $0<\beta \leq \alpha$, i.e., where $\mathcal J=\fin^{\alpha}$. Whenever relevant, we either make an explicit comment or
the reader can substitute $\fin^\alpha$ by $\fin(\vec \phi)$ (but again, do not substitute for the word finite).

\medskip
\begin{mprop}\label{mprop.alphaD}
$\forces_{\M^{\I}_{\alpha}} (\forall y \in \pi[T]) \; y \cap x_{\dot G} \in \fin^{\alpha}$.
\end{mprop}
The Main Proposition will be proved in Section~\ref{s.alphaD.branch} below.
All our results about $\fin^\alpha$ from Theorem~\ref{t.intro.alphaD} follow from the Main Proposition~\ref{mprop.alphaD} as a corollary:
\begin{corollary}
Assuming the Main Proposition~\ref{mprop.alphaD}, Theorem~\ref{t.intro.alphaD} holds.
\end{corollary}
\begin{proof}
It suffices to replace $\M^{\I}$ by $\M^{\I}_\alpha$ in the proofs of Corollaries~\ref{cor.1D.analytic}, \ref{cor.1D.PD}, and \ref{cor.1D.AD} (just as we did in Corollary~\ref{cor.2D} in the two-dimensional case).
\end{proof}

\medskip

\subsection{Properties of the general higher-dimensional forcing}\label{s.alphaD.properties}
Before we prove the Main Proposition~\ref{mprop.alphaD} we collect the necessary facts about $\M^{\I}_\alpha$.

\begin{lemma}\label{l.alphaD.properties}~
\begin{enumerate}
\item\label{l.alphaD.properties.I} For any $A \in \mathcal I$, $\forces_{\M^{\I}_\alpha} x_{\dot G} \cap \check A \in \fin^\alpha$.
\item\label{decomposition.alphaD}
Let $k\in\omega$. 
The partial order $\M^{\I}_{\alpha +1}$ is
isomorphic to the product
$\M^{\I}_{\alpha + 1}(\leq (\emptyset,A))\times (\M_{\alpha})^k$, where $A=\{\vec n\in M^\alpha\mid \vec n(0)\geq k\}$, and by $(\M_{\alpha})^k$ we mean $k$-fold (side-by-side) product of $\alpha$-dimensional Mathias forcing $\M_{\alpha}$. 
If $\alpha$ is a limit ordinal, $\M^{\I}_{\alpha}$ is isomorphic to $\M^{\I}_{\alpha}(\leq(\emptyset,A))\times (\Pi_{i< k}\M_{\alpha_i})$.
\item\label{alphaD.large.x_G} $\forces_{\M^{\I}_\alpha} x_{\dot G}  \in (\fin^\alpha)^{++}$.
\end{enumerate}
\end{lemma}
\begin{proof}
(\ref{l.alphaD.properties.I}) Follows by an obvious density argument.  

\medskip
(\ref{decomposition.alphaD}) First we consider the  successor case. 
Define a map 
\[
\phi\colon \M^{\I}_{\alpha + 1}(\leq (\emptyset,A))\times (\M_{\alpha})^k  \to \M_{\alpha + 1}^{\I}
\]
by
\begin{align*}
\big((b,B),(c_i,C_i)_{i < k}\big) \mapsto
\Big(\bigcup_{i < k} \{i\}\times c_i \cup b, \bigcup_{i < k}\{i\}\times C_i \cup  B \Big).
\end{align*}
For a limit ordinal $\alpha$ the map can be defined in exactly the same way.
Both of these maps are easily seen to be bijective and order preserving.

\medskip
(\ref{alphaD.large.x_G}) This is shown easily by induction, slightly adapting the general case of the proof of \ref{l.2D.properties}(\ref{2D.large.x_G}). We leave this to the reader.
\end{proof}

We shall need a more sophisticated way of decomposing the forcing as a product.

\medskip

Towards this, let us regard $M^\alpha$ and $a$ as trees, ordered by the initial segment relation $\sqsubseteq$.
Given $\vec n \in A$, let us see how we can characterize the ``type'' of $\vec n$ in relation to $a$ with respect to $\leq_{lex}$.

First note that since $a\sqsubseteq_\alpha A$ it is enough to characterize the type of $\vec n\res(\lh(\vec n)-1)$ relative to
the following set of domain sequences
\[
a^*=\big\{\vec n' \res\big(\lh(\vec n')-1\big)\mid \vec n'\in a\big\}
\]
(for if $\vec n$ extends $\vec n^* \in a^*$, $\vec n' <_{lex} \vec n$ for every $\vec n' \in a$ which extends $\vec n^*$).

Let $\vec n_0, \hdots, \vec n_k$ enumerate $a^*$ in lexicographically increasing order,
and let $\vec n_i$ be lexicographically maximal in $a^*$ such that $\vec n_i \leq_{lex} \vec n$.
We then know by $a\sqsubseteq_\alpha A$ that $\vec n$ must have a longer initial segment in common with $\vec n_i$ than it does with $\vec n_{i+1}$, provided $i<k$.

\medskip

Let therefore $\vec m_i$ be the shortest initial segment  of $\vec n_i$ such that $\vec m_i <_{lex} \vec n_{i+1}$ for $i<k$, and let $\vec m_k=\emptyset$.
We have just seen that $\vec m_i \sqsubseteq \vec n$.
Moreover if $j<i$, $\vec m_j \not \sqsubseteq \vec n$ 
(for $\vec m_j <_{lex} \vec n_{j+1} \leq_{lex} \vec n_i$ and so $\vec m_j <_{lex} \vec n$).

\medskip

We have thus shown the following lemma:
\begin{lemma}\label{l.disjoint.union}
Suppose $(a,A) \in \M^\I_\alpha$. Let $\vec n_0, \hdots, \vec n_k$ enumerate 
\[
a^*=\big\{\vec n' \res\big(\lh(\vec n')-1\big)\mid \vec n'\in a\big\}
\] 
in lexicographically ascending order, let $\vec m_k=\emptyset$ and for $i<k$ let $\vec m_i$ be the shortest initial segment of $n_i$ such that $\vec m_i <_{lex} \vec n_{i+1}$ (just as above).

Then for each $\vec n\in A$ there is precisely one $i$ such that $\vec m_i \sqsubseteq \vec n$ and $\vec n_i \leq_{lex} \vec n$
(namely the maximal $i$ such that $\vec n_i \leq_{lex} \vec n$).
\end{lemma}
Technical as the previous lemma may be, it allows us to decompose the forcing as a product in a very useful manner.
\begin{lemma}\label{l.decomposition'.alphaD}
Suppose $(a,A) \in \M^\I_\alpha$, and $\vec m_0, \hdots, \vec m_k$ and $\vec n_0, \hdots, \vec n_k$ are defined as in the previous lemma.
Then
$\M^\I_\alpha\big(\leq (a,A)\big)$ is isomorphic to 
\begin{equation}\label{e.product}
 \bigg(\prod_{i < k} \M_{\delta(\vec m_i)}\Big(\leq\big(a(\vec n_i),A_i(\vec m_i)\big)\Big) \bigg)\times \M^{\I}_{\alpha}\Big(\leq\big(a(\vec n_k),A_k\big)\Big) 
\end{equation}
where 
\[
A_i=A\cap\{\vec n\mid \vec m_i \sqsubseteq \vec n\land \vec n_i \leq_{lex} \vec n\}
\]
for each $i\leq k$.
\end{lemma}
\begin{proof}
The crucial observation is that by Lemma~\ref{l.disjoint.union}, $A$ may be written as a disjoint union
\begin{equation}\label{e.partition}
A = 
\bigcup_{i\leq k} A_i
\end{equation}


\medskip

Define a map $\phi$ from $\M^\I_\alpha\big(\leq (a,A)\big)$ to the forcing in \eqref{e.product}
as follows: For $(b,B)\leq(a,A)$ define 
\begin{equation*}
\phi(b,B)=\big((b\cap A_i)(\vec m_i),(B\cap A_i)(\vec m_i)\big)_{i\leq k}.
\end{equation*}
Using the partition from \eqref{e.partition}, it is straightforward to verify that this map is an isomorphism of partial orders.
\end{proof}

Of course we also have a   
diagonalization
lemma (compare Lemmas~\ref{diagonalization.1D} resp.\ \ref{diagonalizing.A}) for $\M^{\I}_\alpha$.

\begin{lemma}\label{diagonalizing.A.alpha}
Let $(A_k)_{k\in\omega}$ be a sequence from $\I^{++}$ 
satisfying $A_{k+1} \subseteq A_k$ for every
$k\in\omega$. Then there is
$A_{\infty}\in \I^{++}$ such that $A_{\infty} \subseteq^*_{\fin^{\alpha}} A_k$ for every
$k\in\omega$.
\end{lemma}

\begin{proof}
The proof of Lemma~\ref{diagonalizing.A} can be transcribed completely mechanically by replacing $\fin\otimes\fin$ by $\fin^\alpha$ everywhere; we leave this to the reader.
\end{proof}

\medskip

\subsection{The Branch Lemma for general higher dimensions}\label{s.alphaD.branch}

The reader will find that our line of argumentation in this section is remarkably close to that of the previous section;
of course this is only true since the proofs there were written with the general case in mind.

Yet again, the crucial definition is that of an invariant tree, analogous to Definitions~\ref{d.1D.T^x} and \ref{d.2D.tree}.
\begin{definition}\label{d.alphaD.tree}
For $x \subseteq M^{\alpha}$, let
\[
T^x = \{t\in T\mid (\exists w \in \pi[T_{[t]}])\; w \cap x \notin \fin^\alpha\}
\]
\end{definition}
As in Sections~\ref{s.1D.branch} and \ref{s.2D.branch}, it is easy to see that whenever $x \Delta x' \in \fin^\alpha$, $T^x = T^{x'}$.
Moreover Facts~\ref{factsaboutTx}(\ref{factsaboutTx.start})--(\ref{factsaboutTx.end}) hold here as well.

\medskip

We are now ready to state the main lemma of this section.

\begin{blemma}\label{l.alphaD.branch}
$\forces_{\M_{\alpha}^{\I}} |\pi[T^{x_{\dot G}}]| \leq 1$.
\end{blemma}

In keeping with the pattern established in previous sections, we postpone the proof of the Branch Lemma and first give the proof of the Main Proposition~\ref{l.alphaD.branch}, assuming the lemma. 
The proof is verbatim the proof of Main Proposition \ref{mprop.2D} except that we use Lemma~\ref{l.alphaD.properties}(\ref{decomposition.alphaD}) to decompose the forcing; we repeat it for the incredulous reader.
\begin{proof}[Proof of Main Proposition \ref{mprop.alphaD}]
Suppose towards a contradiction that some $p_0\in \M^{\I}_\alpha$ forces that 
there is $A \in \pi[T]^{V[\dot G]}$ with $A \cap x_{\dot G} \notin \fin_\alpha$.
The Branch Lemma~\ref{l.alphaD.branch} lets us choose a name $\dot A$ so that $p_0 \forces \pi[T^{x_{\dot G}}] = \{\dot A\}$.

As in the proof of Main Proposition \ref{mprop.2D} on p.~\pageref{mprop.2D}, we show the following claim:
\begin{claim}\label{c.inV.dimalpha}
There is $q\in \M^{\I}_{\alpha}$ and $A'\in \pi[T]$ such that $q\forces\dot A=\check A'$.
\end{claim}
\begin{proof}[Proof of Claim.]
By the generalized diagonalization lemma (Lemma~\ref{diagonalizing.A.alpha}), it suffices to show that if $p \leq p_0$ and
$p$ decides $\vec n\in \dot A$ then in fact $(a(p_0),A(p))$ decides $\vec n\in\dot A$.

So let us assume $p \forces \vec n\in \dot A$ (if $p \forces \vec n\notin \dot A$ the proof is similar).
We must show that for an arbitrary  $\M^{\I}_{\alpha}$-generic $G$ with $(a(p_0),A(p))\in G$,
it holds that $\vec n\in \dot A^{G}$. 

Fix $k$ large enough so that $\dom(a(p))\subseteq k$.
By Lemma~\ref{l.alphaD.properties}(\ref{decomposition.alphaD}) we can decompose $G$ as $G_0 \times G_1$ where $G_0$ is generic for 
$\prod_{i< k}\M^{\fin^{\alpha_i}}_{\alpha_i}$
and $G_1$ is $\M^{\I}_\alpha$-generic. 
As $x_{G} \Delta x_{G_1} \in \fin^\alpha$, $T^{x_{G}} = T^{x_{G_1}} \in V[G_1]$. 

By absoluteness, $\dot A^{G} \in V[G_1]$ and 
$\pi[T^{x_{G}}]=\{\dot A^{G}\}$ holds in both $V[G]$ and $V[G_1]$.

Since $a(p) \subseteq \prod_{i< k} \{i\}\times M^{\alpha_i}$ we can find $G'_0$ which is $(\prod_{i< k}\M_{\alpha_i},V[G_1])$-generic over $V[G_1]$ so that letting $G'=G'_0 \times G_1$, 
$p \in G'$. 
Again by $\fin^\alpha$-invariance of $T^x$ and by absoluteness, $\dot A^{G'} \in V[G_1]$ and $\pi[T^{x_{G_1}}] = \{\dot A^{G'}\}$ and
so $\dot A^{G'}=\dot A^{G}$ and $\vec n\in\dot A^{G}$. 
\renewcommand{\qedsymbol}{{\tiny Claim \ref{c.inV.dimalpha}.} $\Box$}
\end{proof}
Just as in the proof of Main Proposition \ref{mprop.2D} we conclude that $A'\in \mathcal I$ by absoluteness while 
$q \forces x_{\dot G} \cap \check A' \notin\fin_\alpha$, contradicting Lemma~\ref{l.alphaD.properties}(\ref{l.alphaD.properties.I}).
\renewcommand{\qedsymbol}{{\tiny Main Proposition \ref{mprop.alphaD}.} $\Box$}
\end{proof}

\medskip

Gradually working towards a proof of the Branch Lemma~\ref{l.alphaD.branch},
we start by introducing some notation. Set
\begin{align*}
U = \{(a,t) \in \powerset(M^{\alpha}) \times T \mid a \text{ is finite}\}.
\end{align*}
For $\vec u \in U$, we will often write
$\vec u = (a(\vec u), t(\vec u))$. Define an
ordering $\leq_U$ on $U$ by
\begin{align*}
\vec u_1 \leq_U \vec u_0 \Leftrightarrow a(\vec u_1) \sqsupseteq_{\alpha} a(\vec u_0)
\land t(\vec u_1) \sqsupseteq t(\vec u_0).
\end{align*}

Assume for a moment that $G$ is $\M^{\I}_{\alpha}$-generic over $V$ and work in $V[G]$. For a fixed $x \in \powerset(M^{\alpha})$, define a set $U^x \subseteq U$ consisting of those pairs $(a,t) \in U$ such that there is $w \in [T_{[t]}]$ with
\begin{enumerate}
\item $\pi(w) \cap x_G \notin \fin^{\alpha}$,
\item $\dom_{\alpha}(a) \subseteq \dom_{\alpha}^{\infty}(x \cap \pi(w))$ and
\item $(\forall \vec n \in \dom_{\alpha}(a)) \; a(\vec n) \subseteq x(\vec n) \cap
\pi(w)(\vec n)$.
\end{enumerate}
Note that $U^x$ is closed under initial segments with respect to $\leq_U$, and that an
infinite chain through $U^x$ will give a set $A \in\pi[T]$ with a large
intersection with $x$, and a $(\fin_\alpha)^{++}$-subset of this intersection to witness its largeness in
a useful manner.

In analogy to trees, when $\vec u_0\in U$ we again write $U^{x}_{[\vec u_0]}$ for 
$\{\vec u\in U \mid \vec u_0 \leq \vec u\}$.

Finally working in $V$ again, we note the following about $U^{x_{\dot G}}$:

\begin{lemma} \label{exploringTx.alpha} Suppose $(a,A) \forces \vec u \in U^{x_{\dot G}}$.
\begin{enumerate}
\item \label{a(u)-subset.alpha} It holds that $a(\vec u) \subseteq a$ and moreover,
\begin{align*}
(a(\vec u), A) \forces \vec u \in U^{x_{\dot G}}.
\end{align*}
Similarly, if $(a,A) \forces \vec u' \notin U^{x_{\dot G}}$ and $a(\vec u') \subseteq a$ then
\begin{align*}
(a(\vec u), A) \forces \vec u' \notin U^{x_{\dot G}}.
\end{align*}

\item \label{I-subset.alpha} If $A'\subseteq^*_{\I} A$ such that
$(a,A') \in \M_{\alpha}^{\I}$, then also $(a,A') \forces \vec u \in U^{x_{\dot G}}$.

\item \label{infiniteExtension.terminal.alpha} The set
$A' \subseteq M^{\alpha}$ defined by
\begin{align*}
A'=\{\vec n \mid
  (\exists p' \leq (a,A))(\exists \vec u'\leq_U \vec u)\;
  \vec n\in a(\vec u') \land p'\forces \vec u' \in U^{x_{\dot G}}\}
\end{align*}
is not in $\I$.

\item \label{infiniteExtension.nonterminal.alpha} For a non-empty domain sequence $\vec n \in \dom_\alpha(a(\vec u))$, the set
$A_{\vec n} \subseteq M^{\delta_{\alpha}(\vec n)}$ defined by
\begin{align*}
A_{\vec n}=\{\vec m \mid
  (\exists p' \leq (a,A))(\exists \vec u'\leq_U \vec u)\;
  \vec m\in a(\vec u')(\vec n) \land p'\forces \vec u' \in U^{x_{\dot G}}\}
\end{align*}
is in $(\fin^{\delta_{\alpha}(\vec n)})^{+}$.
\end{enumerate}
\end{lemma}
\begin{proof}

(\ref{a(u)-subset.alpha}) Immediate from the definition of $U^{x_{\dot G}}$.

\medskip
(\ref{I-subset.alpha}) Suppose that $(a,A') \notforces \vec u \in U^{x_{\dot G}}$. Then
there is some $(b,B) \leq (a,A')$ such that $(b,B) \forces \vec u \notin U^{x_{\dot G}}$.
Since $A' \setminus A \in \I$, there is some $B' \subseteq B \cap A$ such that $B' \in
\I^{++}$. However, $(b,B') \leq (b,B)$ and $(b,B')\leq (a,A)$, which is a contradiction.

\medskip
(\ref{infiniteExtension.terminal.alpha}) 
Although the proof is practically identical to that of Lemma~\ref{exploringTx}(\ref{infiniteDomain}), we give the details for the reader's convenience. 
Assume to the contrary that $A' \in \mathcal I$. Then $A \setminus A' \in \I^+$, so take
$B \subseteq A\setminus A'$ such that $B\in \I^{++}$ and set $p = (a,B) \in \M^{\I}_\alpha$.
Since $p \forces \vec u \in U^{x_{\dot G}}$ we can find a name $\dot w$ such that
\[
p \forces \dot w \in \pi[T_{[t(\vec u)]}] \land \dot w\cap x_{\dot G} \notin \fin^\alpha.
\]
(As in Lemma~\ref{exploringTx}(\ref{infiniteDomain}) it would suffice if $p\forces T^{x_{\dot G}}\neq\emptyset$).
Thus we can extend $p$ to $p'$ to force some terminal sequence $\vec n$ into $\dot w\cap x_{\dot G}\setminus a(p)$.
But it has to be the case that $\vec n \in a(p')$. Whence $\vec n \in A'$ by definition of $A'$,
contradicting that also $\vec n \in B$  which is disjoint from $A'$.

\medskip
(\ref{infiniteExtension.nonterminal.alpha}) 
The proof is identical that of Lemma~\ref{exploringTx}(\ref{infiniteExtensionInK}) in essence, but differs substantially in notation.  
Assume to the contrary that $A_{\vec n} \in \fin^{\delta(\vec n)}$. 
Then we can find $p \leq (a,A)$ such that $A(p)(\vec n)$ is disjoint from  $A_{\vec n}$.
Since $p \forces \vec u \in U^{x_{\dot G}}$ we can find a name $\dot w$ such that
\[
p \forces \dot w \in \pi[T_{[t(\vec u)]}] \land  \dot w\cap x_{\dot G} \in (\fin^\alpha)^{+}.
\]
and 
\[
p \forces  \dom_\alpha(a(\vec u))\subseteq \dom_\alpha^\infty(\dot w\cap x_{\dot G}).
\]
Therefore $\vec n \in \dom_\alpha^\infty(\dot w\cap x_{\dot G})$ and we can extend $p$ to $p'$ to force a terminal sequence $\vec n \conc \vec n'$ into $\dot w\cap x_{\dot G}\setminus a(p)$.
But as in the proof of the previous item, it has to be the case that $\vec n\conc \vec n' \in a(p')$, whence $\vec n'  \in A_{\vec n}$ by definition of $A_{\vec n}$,
contradicting that also $\vec n' \in A(p')(\vec n)$ which is disjoint from $A_{\vec n}$.
\end{proof}

Define a set $\Gamma$ as follows:
\begin{align*}
\Gamma= \{(p,\vec u^0, \vec u^1) \in \M_{\alpha}^{\I} \mid
(\forall i \in \{0,1\}) p\forces \vec u^i \in U^{x_{\dot G}}\}.
\end{align*}
Define two orderings on $\Gamma$:
\begin{align*}
(p_1, \vec u_1^0, \vec u_1^1) \leq _{\Gamma} (p_0, \vec u_0^0, \vec u_0^1) \Leftrightarrow
p_1 \leq p_0 \land \vec u_1^i \leq_U \vec u_0^i
\end{align*}
for $i\in\{0,1\}$, and
\begin{align*}
(p_1, \vec u_1^0, \vec u_1^1) \prec _{\Gamma} (p_0, \vec u_0^0, \vec u_0^1) \Leftrightarrow
p_1 \leq p_0 \land \left[a(\vec u_0^0) \cap a(\vec u_0^1) \sqsubset_{\alpha} a(\vec u_1^0)\cap a(\vec
u_1^1)\right].
\end{align*}
Note that $\Gamma$ is well-founded with respect to the second ordering, $\prec_{\Gamma}$.
Indeed, suppose towards a contradiction that there is an infinite sequence $(p_0,\vec u_0^0,
u_0^1) \prec_{\Gamma} (p_1,\vec u_1^0, u_1^1) \prec_{\Gamma} \cdots$. Set
\begin{align*}
y^i = \bigcup_{n\in\omega} t(\vec u_n^i)
\end{align*}
for $i\in\{0,1\}$, and
\begin{align*}
A = \bigcup_{n\in\omega} a(\vec u_n^0) \cap a(\vec u_n^1).
\end{align*}
The sequence is $\prec_{\Gamma}$-decreasing and from $\Gamma$, hence $A \in
(\fin^{\alpha})^{++}$ and $A \subseteq \pi(y^0) \cap \pi(y^1)$, contradicting
$\fin^{\alpha}$-almost disjointness of $\pi[T]$.

The following is the analogue of Lemma~\ref{l.force.finite.bit}, saying that $U^{x_{\dot G}}$ can be approximated reasonably well in the ground model. 
Also as for Lemma~\ref{l.force.finite.bit}, a very similar proof shows that $\M^{\I}_\alpha$ is proper.
\begin{lemma}\label{l.force.finite.bit.alpha}
For each $\vec u_0 \in U$ the set $D(\vec u_0)$ is dense and open in
$\M^{\I}_{\alpha}$, where we define $D(\vec u_0)$ to be the set of $p\in \M^{\I}_{\alpha}$
such that for all $p' \leq p$ and any $\vec u \leq_U \vec u_0 \in U$,
\[
\left[\, p' \forces\, \vec u \in U^{x_{\dot G}}_{[\vec u_0]}\, \right] \Rightarrow
(a(p'),A(p) / a(p')) \forces (\exists t \in T) (a(\vec u),t) \in U^{x_{\dot G}}_{[\vec u_0]}.
\]
\end{lemma}

\begin{proof}
The proof from the two-dimensional case, i.e., of Lemma~\ref{l.force.finite.bit} applies exactly as written once we make the following adaptations:
Firstly replace $\fin\otimes\fin$ by $\fin_\alpha$. 
Secondly, replace $\sqsubset_2$ by $\sqsubset_\alpha$. 
Thirdly, adapt the definition of $B^0_n$ as follows:
\[
B^0_n = \{\vec n\in B^{n-1}_{n-1}\mid (\exists \vec n' \in \dom_\alpha(b^0_n))\; \vec n' \subseteq \vec n \}\cup B^{n-1}_{n-1} \setminus \bigcup \{ C_i \setdef i<n\}.
\]
Then $(b^0_n, B^0_n) \in \M^{\I}_\alpha$, and $B^0_n \cap C_i \in \fin^\alpha$ for each $i<n$.
With these changes, the remainder of the argument for Lemma~\ref{l.force.finite.bit} applies verbatim.
\end{proof}

\medskip

The following technical lemma is crucial to the proof of the branch Lemma. It plays the same role as 
Lemma~\ref{l.1D.properties}(\ref{finitepartnegligible}) and Lemma~\ref{l.force.intersection},
giving us some freedom in tampering with the finite parts of conditions while maintaining that something is forced about $U^{x^{\dot G}}$.

\begin{lemma}\label{l.force.intersection.alpha}
Suppose we are given $\vec u \in U$, $(a,A)\in \M^{\I}_{\alpha}$, and $a' \subseteq a$ so that $(a',A/a') \in \M^{\I}_{\alpha}$, and so that the lexicographically maximal element of $a'$ and is also the lexicographically maximal element of $a$ 
Suppose further
$(a,A) \forces \vec u \in U^{x_{\dot G}}$.  Let $a''
\subseteq a' \cap a(\vec u)$ be arbitrary. Then
\begin{equation*}\label{e.shrink-to-a.alpha}
(a',A/a') \forces  (a'',t(\vec u)) \in U^{x_{\dot G}}.
\end{equation*}
\end{lemma}
\begin{proof}
We will decompose the forcing as a product.
Let $\vec n_0, \hdots, \vec n_k$ and $\vec m_0, \hdots, \vec m_k$ be defined as in Lemma~\ref{l.disjoint.union}.
Then by Lemma~\ref{l.decomposition'.alphaD},
\begin{multline}\label{e.phi}
\M^\I_\alpha\big(\leq (a,A)\big) \cong 
\prod_{i < k} \M_{\delta(\vec m_i)}\Big(\leq\big(a(\vec n_i),A_i(\vec m_i)\big)\Big) \times\\ 
\M^{\I}_{\alpha}\Big(\leq\big(a(\vec n_k),A_k\big)\Big)
\end{multline}
with $A_i$ defined as in the lemma.

Let $D$ consist of those $i < k$ such that 
some element of $a'$ extends $\vec n_i$.
Then writing $A'=A/a'$, Lemma~\ref{l.decomposition'.alphaD} also gives us an isomorphism
\begin{multline}\label{e.phi'}
\M^\I_\alpha\big(\leq (a',A')\big) \cong 
\prod_{i \in D} \M_{\delta(\vec n_i)}\Big(\leq\big(a'(\vec n_i),A'_i(\vec m_i)\big)\Big) \times\\ 
\M^{\I}_{\alpha}\Big(\leq\big(a'(\vec n_k),A'_k\big)\Big)
\end{multline}
with $A'_i$ defined analogously as in the lemma.
We have $A'_i=A_i$ for each $i\in D\cup\{k\}$ and so it is easy to see---e.g., using Lemma~\ref{l.alphaD.properties}(\ref{decomposition.alphaD}), a finite induction, and Lemma~\ref{l.1D.properties}(\ref{isomorphicPO})---that
\begin{equation*}
\M_{\delta(\vec m_i)}\big(\leq(a(\vec n_i),A_i(\vec m_i))\big) \cong \M_{\delta(\vec m_i)}\big(\leq(a'(\vec n_i),A'_i(\vec m_i))\big)
\end{equation*}
for $i\in D$ and
\[
\M^{\I}_{\alpha}\Big(\leq\big(a(\vec n_k),A_k\big)\Big) \cong \M^{\I}_{\alpha}\Big(\leq\big(a'(\vec n_k),A'_k\big)\Big). 
\]
Write
\begin{align*}
\P_{+} =&\prod_{i\in D} \M_{\delta(\vec m_i)}\Big(\leq\big(a(\vec n_i),A_i(\vec m_i)\big)\Big), \\
\P_{-} =&\prod_{i\in D} \M_{\delta(\vec m_i)}\Big(\leq\big(a(\vec n_i),A_i(\vec m_i)\big)\Big), \\ 
\P'_{-}=&\prod_{i\in D} \M_{\delta(\vec m_i)}\Big(\leq\big(a'(\vec n_i),A'_i(\vec m_i)\big)\Big),\\
\P_{\infty} =& \M^{\I}_{\alpha}\Big(\leq\big(a(\vec n),A_k\big)\Big),\\
\P'_{\infty} =& \M^{\I}_{\alpha}\Big(\leq\big(a'(\vec n),A'_k\big)\Big). 
\end{align*}
noting that we have established
\begin{equation}\label{e.isomorphisms1}
\M^\I_\alpha\big(\leq (a',A')\big) \cong  
\P'_{-} \times \P'_\infty  \cong\\
\P_{-} \times \P_\infty
\end{equation}
and
\begin{equation}\label{e.isomorphisms2}
\M^\I_\alpha\big(\leq (a,A)\big) \cong  \P_{+} \times \P_{-} \times \P_\infty
\end{equation}

\medskip

Now finally, let $G'$ be $\M^{\I}_\alpha$-generic over $V$ with $(a',A') \in G'$. 
We must show $(a'',t(\vec u)) \in U^{x_{G'}}$.
Using \eqref{e.isomorphisms1} transform $G'$ into a $\P_{-}\times \P_\infty$ generic $H_{-}\times H_\infty$.
Find a $\P_{+}$-generic $H_{+}$ over $V[H_{-}][H_\infty]$ 
and
let $G$ be the $\M^\I_\alpha$-generic given by $H_{+}\times H_{-} \times H_\infty$ using \eqref{e.isomorphisms2}.
By construction $(a,A) \in G$, whence $(a(\vec u),t(\vec u)) \in U^{x_{G}}$.

\medskip

By definition of $U^{x_G}$ this means that in $V[G]$ we can find $w \in \pi[T_{[t(\vec u)]}]$ 
so that 
\begin{equation}\label{e.existsw.alphaD}
(\exists u \in (\fin^\alpha)^{++})\;a'' \subseteq u \subseteq \pi(w) \cap x_{G}.
\end{equation}
Since $a'' \subseteq x_{G'}$  and since $x_G\Delta x_{G'} \in \fin^\alpha$ we may replace $x_{G}$ by $x_{G'}$ in \eqref{e.existsw.alphaD}, and thus  
\begin{equation}\label{e.x-v.alphaD}
(\exists x \in \pi[T_{[t(\vec u)]}])(\exists u \in (\fin^\alpha)^{++})\; a'' \subseteq  u \subseteq \pi(x) \cap x_{G'}.
\end{equation}
Just as in the two-dimensional case (i.e., the proof of Lemma~\ref{l.force.intersection}) an absoluteness argument easily shows
that \eqref{e.x-v.alphaD} and hence $(a'',t(\vec u)) \in U^{x_{G'}}$ must hold in $V[G']$, proving $(a',A/a')\forces (a'',t(\vec u)) \in U^{x_{\dot G}}$.
\end{proof}

\medskip

After all these preparations, we are finally ready to prove our last and most general instance of the Branch Lemma.

\begin{proof}[Proof of the Branch Lemma \ref{l.alphaD.branch}]
Suppose towards a contradiction we have 
$p\in\M_{\alpha}^{\I}$ and a pair of $\M_{\alpha}^{\I}$-names $\dot
w^0$, $\dot w^1$ such that
\begin{align*}
p \forces (\forall i \in \{0,1\})\; \dot w^i \in \pi[T] \land x_{\dot G}\cap \dot w^i
\notin \fin^{\alpha}
\end{align*}
and $p\forces \dot w^0 \neq \dot w^1$. 
By definition of $\Gamma$ we may find $(p_0, \vec u^0_0, \vec
u^1_0) \in \Gamma$ such that $\pi(t(\vec u^0_0)) \neq \pi(t(\vec u^1_0))$.

\begin{claim}\label{c.alternative.alpha}
There is $(p_1, \vec u^0_1, \vec u^1_1)  \leq_\Gamma (p_0, \vec u^0_0, \vec u^1_0)$, 
a terminal sequence $\vec n^* \in a(\vec u_1^0) \cap a(\vec u_1^1)$, and numbers $l < \lh(\vec n^*)$ and
 and $k^* \in \omega$ such that for any $(p_2, \vec u^0_2, \vec u^1_2)\leq_\Gamma (p_1, \vec u^0_1, \vec u^1_1)$
 and any
\begin{align*}
\vec n \in a(\vec u_2^0) \cap a(\vec u_2^1)
\end{align*}
such that $\vec n^*\res l \sqsubseteq \vec n$, 
we have $\vec n (l) \leq k^*$.
\end{claim}

\begin{proof}[Proof of Claim]
We show that if the claim fails, there is a $\prec_\Gamma$-descending sequence in $\Gamma$.
It suffices to show that any $(p_1, \vec u^0_1, \vec u^1_1) \leq_\Gamma
(p_0, \vec u^0_0, \vec u^1_0)$ has a $\prec_\Gamma$-extension.  So let $(p_1, \vec u^0_1,
\vec u^1_1)
\leq_\Gamma (p_0, \vec u^0_0, \vec u^1_0)$ be given.

Since the claim fails, given any terminal sequence $\vec n \in a(u_1^0)\cap a(u_1^1)$, any $k < \lh(\vec n)$, and 
 \[
(p, \vec u^0,\vec u^1)
\leq_\Gamma (p_1, \vec u^0_1, \vec u^1_1)
\]
we can form an extension
 \[
(q, \vec v^0,
\vec v^1)
\leq_\Gamma (p, \vec u^0, \vec u^1)
\]
such that there is $\vec n' \in a(v_1)\cap a(v_1)$
with $\vec n' \res k = \vec n \res k$ and $\vec n' (k) > \vec n(k)$.

In finitely many steps, construct a (finite) descending sequence 
\[
(p_1, \vec u^0_1, \vec u^1_1) \geq_\Gamma (p_2, \vec u^0_2, \vec u^1_2) \geq_\Gamma\hdots \geq_\Gamma (p_m, \vec u^0_m, \vec u^1_m), 
\]
at each step taking an extension of the previous element as just described. 
We can deal with each $\vec n \in a(\vec u_1^0)\cap a(\vec u_1^1)$ and each $k < \lh(\vec n)$,
so that at the end 
\[
a(\vec u_1^1)\cap a(\vec u_1^1) \sqsubset_\alpha a(\vec u_1^m)\cap a(\vec u_1^m)
\]
Thus we have found $(p_m, \vec u^0_m, \vec u^1_m) \prec_\Gamma (p_1, \vec u^0_1, \vec u^1_1)$.
\renewcommand{\qedsymbol}{{\tiny Claim \ref{c.alternative.alpha}.} $\Box$}
\end{proof}

Let $(p_1, \vec u^0_1, \vec u^1_1)  \leq_\Gamma (p_0,
\vec u^0_0, \vec u^1_0)$, $\vec n^* \in a(\vec u_1^0)\cap a(\vec
u_1^1)$, $l<\lh(\vec n^*)$ and $k^* \in \omega$ be as in the claim.
By Lemma \ref{l.force.finite.bit.alpha} and by replacing $p_1$ by a stronger condition if necessary, we may assume that $p_1 \in D(\vec u_1^0) \cap D(\vec u_1^1)$.  

\medskip

\textbf{Case 1:} Assume first that $l =0$.
Let $A' \subseteq M^{\alpha}$ be defined as in Lemma~\ref{exploringTx.alpha}(\ref{infiniteExtension.terminal.alpha}), namely
\begin{align*}
A'=\{\vec n \mid
  (\exists p' \leq p_1)(\exists \vec u'\leq_U \vec u^0_1)\;
  \vec n\in a(\vec u') \land p'\forces \vec u' \in U^{x_{\dot G}}\}
\end{align*}
By Lemma~\ref{exploringTx.alpha}(\ref{infiniteExtension.terminal.alpha}), $A'\in \mathcal I^{+}$. 
Find $A \subseteq A(p_1)$ 
such that $A \in \I^{++}$, and letting $k^{**}=\max(\dom(a(p_1)))$,
\[
A(p_1) \cap \big( \bigcup_{i \in \dom(p_1)}\{i \leq k^*\}\times M^{\delta(i)} \big) \subseteq A
\]
and 
\[
A \cap \big( \bigcup_{i>k^{**}}\{i\}\times M^{\delta(i)} \big) \subseteq A'.
\]
Letting $p^* = (a(p_1),A)$ we obtain a condition in $\M^{\I}_\alpha$ such that $p^* \leq p_1$. 
Since $p^* \forces \vec u^1_1 \in U^{x_{\dot G}}$, and
we can find $p \leq p^*$, $\vec u$, and $\vec n$ with $\vec n(0)  > k^*, k^{**}$ such that
\[
\vec n \in a(\vec u) \wedge p \forces  \vec u \in U^{x_{\dot G}}_{[\vec u^1_1]}.
\]
It follows that $\vec n\in A'$ and so by definition of $A'$ we can find $p' \leq p_1$
and $\vec u'$ such that
\[
\vec n  \in a(\vec u') \wedge p' \forces  \vec u' \in U^{x_{\dot G}}_{[\vec u^0_1]}.
\]
By extending $p,p'$ if necessary, we can assume that $a(p)$ and $a(p')$ have the same lexicographically maximal element.
As $p_1 \in D(\vec u^0_1) \cap D(\vec u^1_1)$ and $p,p' \leq p_1$,
we can find $\vec u^0$ and $\vec u^1$ such that
\[
\vec n \in a(\vec u^0) \wedge (a(p),A(p_1)/a(p)) \forces  \vec u^0 \in U^{x_{\dot G}}_{[\vec u^0_1]}
\]
and
\[
\vec n \in a(\vec u^1) \wedge (a(p'),A(p_1)/a(p')) \forces  \vec u^1 \in U^{x_{\dot
G}}_{[\vec u^1_1]}.
\]
Let $a = a(p) \cap a(p')$ (whose lexicographically maximal element is also that of $a(p)$ as well as that of $a(p')$).
For each $i\in \{0,1\}$ we have 
\[
a(\vec u_1^i) \cup
\{\vec n\} \subseteq a \subseteq a(p), a(p')
\]
and so
by Lemma~\ref{l.force.intersection.alpha}
\begin{align*}
\big(a, A(p_1)/a)\big) \forces (a(\vec u_1^i) \cup
\{\vec n\}, t(\vec u^i)) \in U^{x_{\dot G}}_{[\vec u_1^i]}
\end{align*}
for each $i\in\{0,1\}$. Letting 
\[
p_2 = \big(a, A(p_1)/a\big)
\]
and 
\[
u_2^i = (a(\vec u_1^i) \cup
\{\vec n\}, t(\vec u^i))
\]
for $i \in \{0,1\}$ we obtain
$(p_2, \vec u_2^0, \vec u_2^1) \leq_\Gamma (p_1, \vec u^0_1, \vec u^1_1)$ with $\vec n \in a(\vec u_2^0) \cap a(\vec u^2_1)$ and 
$\vec n(0) > k^*$, which contradicts the choice of $(p_1, \vec u_1^0, \vec u_1^1)$, $\vec n^*$, and $k^*$.

\medskip

\textbf{Case 2:}
In case $l > 0$, let $\vec n = \vec n^* \res l$ and consider the set $A_{\vec n}$ defined as in Lemma~\ref{exploringTx.alpha}(\ref{infiniteExtension.nonterminal.alpha}), namely
\begin{align*}
A_{\vec n}=\{\vec m \mid
  (\exists p \leq p_1)(\exists \vec u\leq_U \vec u^0_1)\;
  \vec m\in a(\vec u)(\vec n) \land p\forces \vec u \in U^{x_{\dot G}}\}.
\end{align*}
Lemma~\ref{exploringTx.alpha}(\ref{infiniteExtension.nonterminal.alpha}) ensures $A_{\vec
n} \in (\fin^{\delta_{\alpha}(\vec n)})^{+}$. 

Let $k^{**}=\max(\{\vec n(l)\mid\vec n\in a(p_1)\})$.
Find $A \subseteq A(p_1)$ 
such that $A \in (\fin^{\delta(\vec n)})^{++}$, 
\[
A(p_1) (\vec n) \cap \{ \vec m \in \omega^{<\omega} \mid \vec m(0) \leq k^{**}\} \subseteq A 
\]
and 
\[
A(\vec n) \cap \{\vec m \in \omega^{<\omega}\mid \vec m(0) > k^{**}\} \subseteq A_{\vec n}.
\]
By choice of $k^{**}$, letting $p^* = (a(p_1),A)$ we obtain a condition in $\M^{\I}_\alpha$, $p^* \leq p_1$. 

\medskip

Since $p^* \forces \vec u_1^1 \in
U^{x_{\dot G}}$ we can find
$p \leq p^*$ and $\vec u \leq_U \vec u_1^1$ such that there exists a terminal sequence $\vec
m \in a(\vec u)$ with $\vec m(l) > k^*, k^{**}$ and such that
\begin{align*}
p\forces \vec u \in T^{x_{\dot G}}_{[\vec u_1^1]}.
\end{align*}
By definition of $A$ we infer $\vec m = \vec n \conc \vec m'$ for some $\vec m'\in A_{\vec n}$, and so we can find $p'\leq p_1$ and $\vec u' \leq_U \vec u_1^0$ such that
\begin{align*}
\vec m \in a(\vec u')
       \land p'\forces \vec u' \in T^{x_{\dot G}}_{[\vec u_1^0]}.
\end{align*}
Using that $p_1 \in D(\vec u_1^0) \cap D(\vec u_1^1)$ and Lemma~\ref{l.force.intersection.alpha},
argue verbatim as in the previous case to construct
$(p_2, \vec u_2^0, \vec u_21) \leq_\Gamma$ with $\vec m \in a(\vec u_2^0) \cap a(\vec u_21)$.
Since
$\vec m(l) > k^*$, this contradicts the choice of $(p_1, \vec u_1^0, \vec u_1^1)$, $\vec n^*$, $l$, and $k^*$.
\renewcommand{\qedsymbol}{{\tiny Branch Lemma \ref{l.alphaD.branch}.} $\Box$}
\end{proof}


\medskip

\section{Postscript: A dichotomy for Borel ideals?}\label{s.postscript}

The results of this paper show that for $\mathcal J$ in  a rather vast class of Borel ideals in $\omega$, one can prove that there are no definable $\mathcal J$-MAD families, under suitable assumptions on either what definable means, or what background theory is adopted.

It is worth noting that it is \emph{not} the case that such a theorem is true for \emph{every} Borel ideal on $\omega$. Indeed, the ideal on $\omega\times\omega$, defined by
$$
\mathcal J=\{x\subseteq \omega\times \omega: (\forall n\in\omega) \{m: (n,m)\in x\}\text{ is finite}\}
$$
clearly admits the $\mathcal J$-MAD family, namely $\{\{i\}\times\omega: i\in\omega\}$.

It remains an interesting open problem if it is possible to characterize the Borel ideals for which an analogue of Theorem \ref{t.intro.alphaD} is true, and for which that type of theorem fails. In other words:

\begin{question}
Is there a dichotomy for Borel ideals on $\omega$ which characterizes when there are/are no definable MAD families with respect to a given Borel ideal?
\end{question}

We note that in the more general setting of finding maximal discrete sets for Borel graphs, a result of Horowitz and Shelah \cite{horowitz-shelah-graphs} shows there is no reasonable dichotomy. It is, however, not clear that the obstruction found there also apply to the special case of $\mathcal J$-mad families, when $\mathcal J$ is a Borel ideal.

\bibliographystyle{amsplain}
\bibliography{MADfamiliesunderPD}

\end{document}